\documentclass[a4paper,12pt,twoside]{article}
\usepackage[ansinew]{inputenc}
\usepackage{amsfonts}
\usepackage{latexsym}
\usepackage{amsmath}
\usepackage{amssymb}
\usepackage{eepic} 

\newcommand{\veps}{\varepsilon}
\newcommand{\R}{\mathbb{R}}

\newcommand{\C}{\mathbb{C}}
\newcommand{\N}{\mathbb{N}}

\newcommand{\vp}{\varphi}

\newtheorem{defin}{Definition}

\newtheorem{theorem}[defin]{Theorem}

\newtheorem{exa}{Example}
\newenvironment{example}{\begin{exa}\rm}{\end{exa}}
\newtheorem{lemma}[defin]{Lemma}
\newenvironment{proof}
{\noindent{\it Proof.}}{\hfill $\Box$\par\vspace{2.5mm}}
\newtheorem{rem}{Remark}
\newenvironment{remark}{\begin{rem}\rm}{\end{rem}}

\numberwithin{equation}{section}

\makeatletter
\renewcommand{\ps@myheadings}{%
\renewcommand{\@evenhead}%
{{\rm\thepage}\hfil{\sc Gundersen, Heittokangas, Wen}\hfil}%
\renewcommand{\@oddhead}%
{\hfil{{\sc Nevanlinna and Valiron deficient values}\hfil{\rm\thepage}}}%
\renewcommand{\@evenfoot}{}%
\renewcommand{\@oddfoot}{}%
}\makeatother 

\renewenvironment{thebibliography}[1]{%
\begin{oldthebibliography}{#1}%
\setlength{\parskip}{0pt}%
\setlength{\itemsep}{0pt}%
}%
{%
\end{oldthebibliography}%
}

\title{\bf\Large Deficient values of solutions of \\ linear differential equations}
\author{\textbf{Gary G.~Gundersen$^1$, Janne Heittokangas$^{2,}$\footnote{Corresponding author.}$\, $ and Zhi-Tao Wen$^3$}\\[2pt]
{\footnotesize $^1$University of New Orleans, Department of Mathematics, New Orleans,}\\[-5pt]
{\footnotesize LA 70148, USA; ggunders@uno.edu}\\[2pt]
{\footnotesize $^{2,3}$Taiyuan University of Technology, Department of Mathematics,}\\[-5pt]
{\footnotesize Yingze West Street, No.~79, Taiyuan 030024, China}\\[2pt]
{\footnotesize $^2$\emph{Current address:} University of Eastern Finland, Department of Physics and Mathematics,}\\[-5pt]
{\footnotesize P.O.~Box 111, 80101 Joensuu, Finland; janne.heittokangas@uef.fi}\\[2pt]
{\footnotesize $^3$\emph{Current address:} Shantou University, Department of Mathematics,}\\[-5pt]
{\footnotesize Daxue Road No.~243, Shantou 515063, China; zhtwen@stu.edu.cn}}
\date{}

\begin{document}
\maketitle

\begin{abstract}
Differential equations of the form $f'' + A(z)f' + B(z)f = 0$ (*) are considered,
where $A(z)$ and $B(z) \not\equiv 0$ are entire functions.
The Lindel\"of function is used to show that for any $\rho \in (1/2, \infty)$, there exists an
equation of the form (*) which possesses a solution $f$ of order $\rho$ with a Nevanlinna
deficient value at $0$, where $f, A(z), B(z)$ satisfy a common growth condition. It is known
that such an example cannot exist when $\rho \leq 1/2$. For smaller growth functions, a geometrical modification of an example of Anderson and Clunie is used to show that for any $\rho \in (2, \infty)$, there exists an equation of the form (*) which possesses a solution $f$ of logarithmic order $\rho$ with a
Valiron deficient value of at $0$, where $f, A(z), B(z)$ satisfy an analogous growth condition. This result is essentially sharp. In both proofs, the separation of the zeros of the indicated solution plays a key role. Observations on the deficient values of solutions of linear differential equations are also given, which include a discussion of Wittich's theorem on Nevanlinna deficient values, a modified Wittich theorem for Valiron deficient values, consequences of Gol'dberg's theorem, and examples to illustrate possibilities that can occur. \\


\end{abstract}


\section{Introduction} \label{intro}


The solutions of the linear differential equation 
	\begin{equation}\label{lde}
	f^{(n)}+A_{n-1}(z)f^{(n-1)}+\cdots +A_1(z)f'+A_0(z)f=0
	\end{equation}
with entire coefficients $A_0(z),\ldots,A_{n-1}(z)$, $A_0(z)\not\equiv 0$,  are entire, and it is 
well known that the zeros of any solution $f\not\equiv 0$ of \eqref{lde} are of multiplicity $\leq n-1$.
The main focus of this paper is on solutions of \eqref{lde} which have less than the usual frequency of zeros. A standard measurement of the frequency of a $c$-point ($c\in\C$) of an entire function $f$ is the Nevanlinna deficiency $\delta_N(c, f)$ defined by 
$$\delta_N(c,f) = \liminf_{r\to\infty} \frac{m(r,f,c)}{T(r,f)}=1-\limsup_{r\to\infty}\frac{N(r,f,c)}{T(r,f)}.$$
If $\delta_N(c, f) > 0$, then $c$ is said to be a {\it Nevanlinna deficient value} of $f$. 

Let $\rho(f)$ denote the order of an entire function $f$. It is known \cite[p.~207]{GO} that an entire function $f$ cannot possess a finite Nevanlinna deficient value when $\rho(f) \leq 1/2$. 
Thus it can be asked:
\begin{quote}
\emph{For any given $\rho \in (1/2, \infty)$, does there exists an equation of the form \eqref{lde} 
that possesses a solution $f$ satisfying $\rho(f) = \rho$ and $\delta_N(0, f) > 0$? }
\end{quote}

If there are no restrictions on the orders of the coefficients $A_j(z)$ in \eqref{lde}, then Gol'dberg's theorem could be used to easily answer this question affirmatively, see the sentence following \eqref{double ineq} below. By putting a common growth restriction on the coefficients, the following result answers this question for second order equations
	\begin{equation}\label{ldeAB}
	f''+A(z)f'+B(z)f=0,
	\end{equation}
where $A(z)$ and $B(z)\not\equiv 0$ are entire functions.

\begin{theorem}\label{main1}
For any given $\rho\in (1/2,\infty)$, there exists an equation of the form \eqref{ldeAB} which
possesses a solution $f$ satisfying $\delta_N(0,f)>0$ and $\rho=\rho(f)\geq\rho(A)\geq\rho(B)$.
\end{theorem}

Theorem~\ref{main1} is sharp with respect to $\rho \in (1/2, \infty)$ because, as noted above, the result does not hold for $\rho \leq 1/2$. The solution $f$ in the proof of Theorem~\ref{main1} is the classical Lindel\"of function $L_{\rho}$ in \eqref{L} below. Thus the crux of the proof is to find entire
coefficients $A(z)$ and $B(z)$ satisfying \eqref{ldeAB} and
	\begin{equation} \label{double ineq}
	\rho(f) \geq \rho(A) \geq \rho(B).
	\end{equation}

If the inequalities \eqref{double ineq} were removed from the conclusion of Theorem~\ref{main1}, then Theorem~\ref{main1} would easily follow by combining Gol'dberg's theorem with the known properties that $L_\rho$ has all simple zeros and $\delta_N(0,L_\rho)>0$, see Sections~\ref{multiplicities-sec} and \ref{proof-main1-sec}. The proof of Theorem~\ref{main1} involves proving a separation of zeros property of $L_{\rho}$, namely, that the zeros of $L_{\rho}$ are uniformly $q$-separated for any $q>\rho$.

Regarding \eqref{double ineq}, recall that \eqref{ldeAB} cannot possess a nontrivial solution of finite order in the case when $\rho(A)<\rho(B)$, see \cite[Theorem~2]{Gundersen2}. Many examples in the literature of solutions of \eqref{ldeAB} involve elementary functions, where $A(z), B(z), f$ are of integer order and satisfy the double inequalities \eqref{double ineq}. However, there are examples of $A(z), B(z), f$ that are of integer order and still do not satisfy \eqref{double ineq}, as Examples~\ref{w-example} and~\ref{Bank-ex} below show. In addition,  Theorem~\ref{main1} addresses all the orders in the infinite interval $(1/2,\infty)$, not just the integer orders. 

\begin{example}\label{w-example}
If $w(z)$ is any entire function, then $f(z)=e^z$ satisfies
$$f''+w(z)f'-(1+w(z))f=0.$$
\end{example}

For smaller growth functions $f$, we obtain an analogous result to Theorem~\ref{main1} by considering Valiron deficient values and logarithmic order. The Valiron deficiency $\delta_V(c, f)$ of a $c$-point of an entire function $f$ is defined by 
$$\delta_V(c,f) = \limsup_{r\to\infty} \frac{m(r,f,c)}{T(r,f)}=1-\liminf_{r\to\infty}\frac{N(r,f,c)}{T(r,f)}.$$
Clearly $0\leq \delta_N(c,f)\leq \delta_V(c,f)\leq 1$. If $\delta_V(c, f) > 0$, then $c$ said to be a {\it Valiron deficient value} of $f$.

For slowly growing entire functions, finite Nevanlinna deficient values are not possible but Valiron deficient values are possible. The growth of such functions $f$ can be measured in terms of the logarithmic order $\rho_{\log} (f)$ defined by
\begin{equation*}
\rho_{\log} (f)=\limsup_{r\to\infty}\frac{\log T(r,f)}{\log\log r}.
\end{equation*}
Observe that finite logarithmic order implies zero order, nonconstant polynomials have logarithmic order one, and that there do not exist any nonconstant entire functions of logarithmic order $<1$, see \cite{BP, Chern}.

A classical result of Valiron \cite{AC, Valiron} says that any entire function $f$ satisfying $T(r,f)=O\left(\log^2 r\right)$
has no finite Valiron deficient values. Entire functions $f$ satisfying $\rho_{\log} (f) < 2$ have this growth rate. Hence we state our main result regarding Valiron deficient values as follows.

\begin{theorem}\label{main2}
For any given $\rho\in (2,\infty)$, there exists an equation of the form \eqref{ldeAB} which
possesses a solution $f$ satisfying $\delta_V(0,f)=1$ and $\rho=\rho_{\log}(f)\geq\rho_{\log}(A)\geq\rho_{\log}(B)$,
where $A(z)$ is transcendental.
\end{theorem}

Theorem~\ref{main2} is essentially sharp with respect to $\rho \in (2, \infty)$ because the result 
does not hold when $\rho < 2$, and the only unsettled case is when $\rho_{\log}(f) = 2$ and $f$ has infinite logarithmic type.

The solution $f$ in Theorem~\ref{main2} is a laborious modification of a function due to Anderson and Clunie \cite{AC}, and its symmetric geometric construction takes a substantial portion of this paper. More precisely, the proof of \cite[Theorem~2]{AC} involves a canonical product with negative real zeros of unbounded multiplicity. Such a function is a real entire function (real on reals), but it cannot be a solution of \eqref{ldeAB} because it has zeros with unbounded multiplicities. Thus we modify the reasoning in \cite{AC} in such a way that the revised canonical product $f$ has only simple zeros lying symmetrically in the left half-plane, where $f$ has the pre-given logarithmic order. The zeros of $f$ will be pairs of complex conjugates so that $f$ becomes real entire, and, in addition, $f(r)=M(r,f)$ holds. These properties are crucial for proving that $\delta_V(0,f)=1$. A  consequence of the proof of Theorem~\ref{main2} is that the zeros of the modified Anderson-Clunie function $f$ are uniformly logarithmically $q$-separated for any $q > \rho_{\log}(f) - 1$.  

\medskip

After preparations in Section~\ref{q-sepa-sec}, we prove Theorem~\ref{main1} in Section~\ref{proof-main1-sec}, while after  preparations in Section~\ref{uniform-logarithmic}, we prove Theorem~\ref{main2} in Section~\ref{proof-main2-sec}. For convenience, Section~\ref{sep zeros} contains information about the separation of zeros of entire functions. 

In Section~\ref{multiplicities-sec}, we discuss the possibilities for the sets $E_N(f)$ and $E_V(g)$ of the Nevanlinna and Valiron deficiencies of solutions $f,g$ of equations of the general form \eqref{lde}. It is shown how it follows from Gol'dberg's theorem that $E_N(f)$ can be countably infinite and $E_V(g)$ can be uncountable. The next section contains examples that illustrate a classical theorem of Wittich on Nevanlinna deficient values and a particular modified Wittich theorem for Valiron deficient values.


\section{Wittich's theorem}  \label{W}


On the topic of possible deficiencies of solutions of equations of the form \eqref{lde}, we recall the following well-known result of Wittich on Nevanlinna deficient values.

\bigskip
\noindent
\textbf{Wittich's theorem.} (\cite[Theorem~4.3]{Laine}, \cite{Wittich})
\emph{Suppose that a solution $f$ of \eqref{lde} is admissible in the sense that
\begin{equation}\label{admissible}
T(r,A_j)=o(T(r,f)),\quad j=0,\ldots,n-1,
\end{equation}
as $r\to\infty$ outside a possible exceptional set of finite linear measure. Then $0$ is the only possible finite Nevanlinna deficient value for $f$. In particular, this is true for transcendental solutions of \eqref{lde} with polynomial coefficients.}\\

The assumption on admissibility of $f$ cannot be removed in Wittich's theorem.
The next example shows that if the growth of at least one of the coefficients in \eqref{lde} is at least that of a solution $f$, then any $c\in\C$ can be a Picard value of $f$. 

\begin{example}\label{any c}
For an arbitrary $c\in\C$ and an arbitrary entire function $w(z)$, the function $f(z)=e^z+c$ with $c$ as its Picard value solves the equation
$$f^{(4)} + ce^{-z}f''' + w(z)f'' - w(z)f' - f = 0.$$
\end{example}

The following example shows that $0$ may or may not be a deficient value for an admissible solution. 

\begin{example}
The function $f(z)=\exp\left(z^2/2\right)$ is an admissible solution of
    $$
    f''' + e^zf'' + (e^z -ze^z - 1 - z^2)f' - (ze^z + e^z + 2z)f = 0,
    $$
which has $0$ as a deficient value. On the other hand, from
\cite[Example~5]{G-S-W}, the function
$g(z)=\exp\left(z^2\right)+e^z$ is an admissible solution of
    \begin{equation}\label{4th-order}
    \begin{split}
    f^{(4)} &+ (8z^3-13)f''-(16z^4+16z^3+12z^2+4z+2)f'\\
    &+ (16z^4+8z^3+12z^2+4z+14)f=0.
    \end{split}
    \end{equation}
We have $T(r,g)\sim N(r,1/g)$ as $r\to\infty$ by \cite[Satz~1-2]{Stein1}, so that $\delta_N(0,g)=0$. Observe that $f_1(z)=\exp\left(z^2\right)$ and $f_2(z)=e^z$ also satisfy \eqref{4th-order}, where $\delta_N(0,f_1)=\delta_N(0,f_2)=1$.
\end{example}

The next example gives an equation of the form \eqref{ldeAB} whose nontrivial solutions are all
admissible solutions with $0$ as a deficient value. 

\begin{example}
The functions $f(z)=\exp\left(z^2/2\right)\sin z$ and
$g(z)=\exp\left(z^2/2\right)\cos z$ discussed in \cite[p.~416]{Gundersen2} 
are linearly independent admissible solutions of
    \begin{equation}\label{example-equation}
    f''-2zf'+z^2f=0.
    \end{equation}
It follows that all the nontrivial solutions of \eqref{example-equation} are admissible and have $0$ as a deficient value. Moreover, if $w(z)$ is any entire function,
then $f$ and $g$ satisfy
    $$
    f''' + (w(z) - 2z)f'' + \left(z^2 - 2 - 2zw(z)\right)f' + \left(z^2w(z) + 2z\right)f = 0.
    $$
More examples of this kind can be generated by using \cite[Example 4]{Gundersen2}.
\end{example}

Next we note that the above examples for Nevanlinna deficient values are also examples for Valiron deficient values, since we
always have $\delta_N(c,f)\leq\delta_V(c,f)$. Although the set $E_N(f)$ is at most countable from the second fundamental theorem, the set $E_V(f)$ can be uncountable. A classical result of Ahlfors-Frostman \cite[p.~276]{Nevanlinna} shows that $E_V(f)$ always has logarithmic capacity zero. Improvements of this result are due to Hyllengren \cite{Hylle} in the finite order case and to Hayman \cite{Hayman2} in the infinite order case. See also \cite[Chapter~4]{GO}.

Wittich's theorem can easily be modified to concern Valiron deficient values. To achieve this, we need to avoid all exceptional sets, and thus the reasoning works only for finite-order solutions. This modified result is particularly valuable when $\rho(f) \leq 1/2$, as we know that $f$ does not have finite Nevanlinna deficient values in this case. Solutions of zero-order are also possible, provided that at least one of the coefficients is transcendental \cite{G-S-W}.

\medskip
\noindent
\textbf{Modified Wittich's theorem.}
\emph{Suppose that a finite-order solution $f$ of \eqref{lde}
satisfies \eqref{admissible} as $r\to\infty$ without an exceptional set.
Then $0$ is the only possible finite Valiron deficient value for $f$.}

\medskip

\begin{proof}
Let $c\in\C\setminus\{0\}$, and write \eqref{lde} in the form
    $$
    (f-c)^{(n)}+A_{n-1}(z)(f-c)^{(n-1)}+\cdots+A_0(z)(f-c)=-cA_0(z).
    $$
Using the standard lemma on the logarithmic derivative, the first main theorem,
and the assumption \eqref{admissible} without an exceptional set, we obtain from
    $$
    \frac{1}{f-c}=-\frac{1}{cA_0(z)}\left(A_0(z)+A_1(z)\frac{(f-c)'}{f-c}
    +\cdots+\frac{(f-c)^{(n)}}{f-c}\right)
    $$
that
    $
    m(r,f,c)=O\left(\log r\right)+o(T(r,f)),
    $
as $r\to\infty$ without an exceptional set. The property \eqref{admissible} guarantees that $f$ is transcendental, even if the coefficients are polynomials. Thus $m(r,f,c)=o(T(r,f))$ as $r\to\infty$ without an exceptional set. This proves that $c$ is not a Valiron deficient value of $f$.
\end{proof}

Theorems~\ref{main1} and \ref{main2} address a different question than that in Wittich's theorem and the modified Wittich's theorem. That said, we mention for independent interest that the proofs of Theorems~\ref{main1} and \ref{main2} do not reveal whether the indicated solution $f$ is an admissible solution of \eqref{ldeAB} or not.


\section{Gol'dberg's theorem and sets of deficiencies}\label{multiplicities-sec}


Any zero of a nontrivial solution of \eqref{lde} must have multiplicity $\leq n - 1$. 
Conversely, we have the following result.

\bigskip
\noindent
\textbf{Gol'dberg's theorem.} (\cite[p.~300]{problembook})
\emph{Let $f\not\equiv 0$ be an entire function whose zeros all have multiplicity at most $n - 1$, $n\in\N$. Then $f$ is a solution of some differential equation of the form \eqref{lde}.}

\bigskip
If $n=1$, then $f$ has no zeros, and the proof of Gol'dberg's theorem is trivial. For the convenience of the reader, we prove the case $n=2$.

\bigskip
\noindent
\emph{Proof of Gol'dberg's theorem for $n=2$.} Let $f\not\equiv 0$ be any entire function whose zeros are all simple. We assume that $f$ has at least one zero, since otherwise the proof is trivial. 

We construct entire functions $A(z)$ and $B(z)\not\equiv 0$ such that $f$ solves \eqref{ldeAB}. For $A(z)$ to be entire, at the zeros $z_k$ of $f$, $A(z)$ needs to solve the interpolation problem
	\begin{equation}\label{interpolation}
	A(z_k)=-f''(z_k)/f'(z_k)=\sigma_k,
	\end{equation}
where $\sigma_k\in\C$. Note that \eqref{interpolation} can always be solved: If $\{z_k\}$ is a finite sequence, then $A(z)$ can
be chosen to be the Lagrange interpolation polynomial, while if $\{z_k\}$ is an infinite sequence, then $A(z)$ can be constructed by means of Mittag-Leffler series. Let $\zeta\neq z_k$ be fixed. Along with \eqref{interpolation}, we may require that
    \begin{equation}\label{interpolation2}
	A(\zeta)\neq -f''(\zeta)/f'(\zeta).
	\end{equation}
This guarantees that $f''(z)+A(z)f'(z)\not\equiv 0$.
After an entire $A(z)$ satisfying \eqref{interpolation} and \eqref{interpolation2} has been found,
we define $B(z)$ by
	\begin{equation}\label{def-B}
	B(z)=-(f''(z)+A(z)f'(z))/f(z),
	\end{equation}
which is entire and $\not\equiv 0$. This completes the proof. \hfill$\Box$

\smallskip

\begin{remark}
If $f\not\equiv 0$ is any entire function, then for a suitable constant $c$, the function $g=f-c$ has all simple zeros. A consequence of this easy observation is that many properties of $f$, such as the number of deficient values, remain valid for $g$, and by Gol'dberg's theorem, $g$ solves some equation of the form \eqref{ldeAB}.
\end{remark}

\begin{example}\label{several-defvalues-ex}
(1) Solutions of \eqref{ldeAB} may have any pre-given finite number $q\geq 2$ of Nevanlinna
deficient values. Indeed, set
	$$
	f(z)=\int_0^ze^{-\zeta^q}\, d\zeta\quad\textnormal{and}\quad
	a_k=e^{2\pi ki/q}\int_0^\infty e^{-\zeta^q}\, d\zeta,
	$$
where $k=1,\ldots,q$. Then $f$ is entire, $\delta_N(a_k,f)=1/q$, and $\delta_N(c,f)=0$ whenever $c \not= a_k$ ($k=1,\ldots,q$), see \cite[pp.~46--47]{Hayman}. Since $f'(z)=e^{-z^q}$ has no zeros, we obtain that for any $c \not = a_k$, ($k = 1, \ldots , q$), the function $g = f - c$ has all simple zeros and exactly $q$ Nevanlinna deficient values, and by Gol'dberg's theorem, $g$ is a solution of some equation of the form \eqref{ldeAB}.

(2) Eremenko \cite[p.~132]{GO} proved that for any countable set $E\subset\C$ and any $\rho>1/2$, there exists an entire function $f$ of order $\rho$ for which $E_N(f)=E$. If $f$ is any such function, then for a suitable $c\in\C$, the function $g=f-c$ has only simple zeros and countably many Nevanlinna deficient values, and  $g$ solves an equation of the form \eqref{ldeAB}.

(3) Let $f$ be an entire function with uncountably many Valiron deficient values \cite[p.~118]{GO}. For a suitable $c\in\C$, $g=f-c$ has only simple zeros and uncountably many Valiron deficient values, and $g$ solves an equation of the form \eqref{ldeAB}.
\end{example}

Gol'dberg's theorem does not give information about the orders of the coefficients in \eqref{lde}. Thus, although Gol'dberg's theorem is useful in the above discussions, it cannot be used to prove the respective inequalities $\rho(f)\geq\rho(A)\geq\rho(B)$ and $\rho_{\log}(f)\geq\rho_{\log}(A)\geq\rho_{\log}(B)$ in Theorems~\ref{main1} and \ref{main2}.


\section{Separation of zeros} \label{sep zeros}


Even though the double inequalities $\rho(f)\geq\rho(A)\geq\rho(B)$ in Theorem~\ref{main1} occur for solutions of \eqref{ldeAB}, they do not always hold, as many examples show, including Examples~\ref{w-example} and \ref{Bank-ex}. 

\begin{example}\label{Bank-ex}
Let $\{z_n\}$ be the sequence defined by $z_{2n-1}=2^n$ and $z_{2n}=2^n+\veps_n$,
where $\{\veps_n\}$ is any fixed sequence satisfying
	$$
	0<\veps_n<\exp\left(-\exp\left(2^n\right)\right),\quad n\geq 1.
	$$
Thus the sequence $\{z_n\}$ has non-zero distinct points, and its exponent of convergence
is equal to zero (see Section~\ref{q-sepa-sec} below). Then the canonical product
	\begin{equation}\label{fun}
	f(z)=\prod_{n=1}^\infty\left(1-\frac{z}{z_n}\right)
	\end{equation}
is an entire function of order zero. Moreover,
    \begin{eqnarray}
    f'(z_k)&=&-\frac{1}{z_k}\prod_{n\neq k}\left(1-\frac{z_k}{z_n}\right),\label{first-derivative}\\
    f''(z_k)&=&\frac{2}{z_k}\sum_{m\neq k}\frac{1}{z_m}\prod_{j\neq k,m}\left(1-\frac{z_k}{z_j}\right),\nonumber
    \end{eqnarray}
so that
    \begin{equation}\label{target}
    \sigma_k=-\frac{f''(z_k)}{f'(z_k)}=2\sum_{n\neq k}\frac{1}{z_n-z_k}.
    \end{equation}
Then the reasoning in the proof of \cite[Corollary~1]{Bank} shows that
	$$
	|\sigma_{2k-1}|\geq\exp\left(\exp\left(|z_{2k-1}|\right)\right)+O(1).
	$$
Thus no finite order $A(z)$ can satisfy \eqref{interpolation}, even though $\rho(f)=0$.
Like before, we set $B(z)$ to be the function in \eqref{def-B}, and then $\rho(B)=\infty$.
\end{example}

Obviously, many zeros of the function $f$ in \eqref{fun} are close together. 
As stated in Section~\ref{intro}, the separation of zeros of the indicated solutions play
a key role in the proofs of Theorems~\ref{main1} and~\ref{main2}. For illustrative purposes, 
we discuss some examples regarding uniformly $q$-separated sequences (defined below).

First we recall a few concepts from \cite[Chapter V]{Tsuji}.
We say that an infinite sequence $\{z_n\}$ in $\C\setminus\{0\}$
with no finite limit points has a finite exponent of convergence $\lambda> 0$ if
$\{1/|z_n|\}\in\ell^{\lambda+\veps}\setminus\ell^{\lambda-\veps}$
for any $\veps\in (0,\lambda)$, while $\lambda=0$ if $\{1/|z_n|\}\in\ell^\veps$ for any $\veps>0$.
The genus of $\{z_n\}$ is the unique integer $p\geq 0$ satisfying
$\{1/|z_n|\}\in\ell^{p+1}\setminus\ell^{p}$.
If $\lambda\not\in\N\cup\{0\}$, then $p=\lfloor \lambda\rfloor$ ($=$ the integer part of $\lambda$),
while if $\lambda\in\N\cup\{0\}$, then either $p=\lambda$ or $p=\lambda-1$. In all
cases, $p\leq\lambda$. The Weierstrass convergence factors are
    $$
    e_0(z)=1\quad\textnormal{and}\quad e_k(z)=\exp\left(\sum_{j=1}^k\frac{z^j}{j}\right),
    $$
where $k\in\N$. If $\{z_n\}$ has finite genus $p\geq 0$, then the canonical product
    \begin{equation}\label{product-f}
    f(z)=\prod_{n=1}^\infty\left(1-\frac{z}{z_n}\right)e_p\left(\frac{z}{z_n}\right)
    \end{equation}
converges uniformly in compact subsets of $\C$, and hence represents an entire
function having zeros precisely at the points $z_n$. We have $p\leq\lambda\leq \rho(f)$.

Following an analogous definition in the unit disc \cite{GH}, we say that a sequence $\{z_n\}$
of finite genus $p$ is \emph{uniformly $q$-separated}
for $q\geq 0$ provided that there exists a constant $C>0$ such that
	\begin{equation}\label{separation}
	\inf_{k\in\N}\left\{e^{C|z_k|^q}
    \prod_{n\neq k}
    \left|1-\frac{z_k}{z_n}\right|\left|e_p\left(\frac{z_k}{z_n}\right)\right|\right\}>0.
	\end{equation}
An elementary differentiation of \eqref{product-f} followed by a substitution $z=z_k$ yields
    \begin{equation}\label{first-derivative2}
    \begin{split}
    f'(z) &= -\sum_{j=1}^\infty\frac{z^p}{z_j^{p+1}}e_p\left(\frac{z}{z_j}\right)
    \prod_{n\neq j}\left(1-\frac{z}{z_n}\right)e_p\left(\frac{z}{z_n}\right),\\
    f'(z_k) &=-\frac{e_p(1)}{z_k}
    \prod_{n\neq k}\left(1-\frac{z_k}{z_n}\right)e_p\left(\frac{z_k}{z_n}\right),
    \end{split}
    \end{equation}
so that we may write \eqref{separation} equivalently as
	\begin{equation}\label{equivalent}
	\inf_{k\in\N}\left\{|z_k|e^{C|z_k|^q}|f'(z_k)|\right\}>0.
	\end{equation}
The definition of a $q$-separated sequence in \cite{HL} assumes that $C=1$ in \eqref{equivalent}.

\begin{example}\label{q-example}
We prove that the zeros of the function $f$ in \eqref{fun} are not uniformly
$q$-separated for any $q\geq 0$. Let $q\geq 0$, and set $k=2n-1$. We have
    \begin{eqnarray*}
    |z_k||f'(z_k)| &=& \left|1-\frac{z_k}{z_{k+1}}\right|\prod_{j\neq k,\,k+1}
    \left|1-\frac{z_k}{z_j}\right|\\
    &=& \frac{z_{k+1}-z_k}{z_{k+1}}\cdot\exp\left(\sum_{j\neq k,\,k+1}
    \log\left|1-\frac{z_k}{z_j}\right|\right)\\
    &\leq & \frac{\veps_n}{2^n+\veps_n}\cdot\exp\left(\sum_{j\neq k,\,k+1}
    \log\left(1+\left|\frac{z_k}{z_j}\right|\right)\right)\\
    &\leq & \frac{\veps_n}{2^n}\cdot\exp\left(\sum_{j=1}^\infty
    \left|\frac{z_k}{z_j}\right|\right)\leq\frac{\veps_n}{2^n}\cdot\exp\left(K\left|z_k\right|\right)
    \end{eqnarray*}
for some constant $K>0$ independent of $k$. Thus, for every $C>0$,
    $$
    |z_k|e^{C|z_k|^q}|f'(z_k)|\leq\exp\Big(C2^{nq}+K2^n-n\log 2-\exp\left(2^n\right)\Big)\to 0
    $$
as $n\to\infty$ (or equivalently as $k\to\infty$).
\end{example}

\begin{remark}
In Example~\ref{Bank-ex}, we showed that the function $f$ in \eqref{fun} could not satisfy the conclusion in Theorem~\ref{main1}, and above we showed that the zeros of this function are not uniformly $q$-separated for any $q\geq 0$. In contrast, the zeros of the
indicated solution $f$ in Theorem~\ref{main1} (the Lindel\"of function $L_{\rho}$) are uniformly $q$-separated for every $q>\rho(f)$, see Section~\ref{proof-main1-sec}.
\end{remark}

For completeness, we construct an example of a uniformly $q$-separated sequence for $q>0$ and $\lambda=0$.

\begin{example}\label{ex1}
The sequence $\{z_n\}$ given by $z_n=2^n$ is uniformly $0$-separated and has
zero exponent of convergence \cite[p.~299]{HL}. For a fixed $q>0$, choose 
$\gamma_n\in\left[\min\left\{1/2, 2^n\exp\left(-2^{nq}\right)\right\},1\right)$,
and define $w_n=z_n+\gamma_n$. Then $\{w_n\}$ is also uniformly $0$-separated
and has zero exponent of convergence. Let $\{\zeta_n\}$ denote the union sequence 
$\{z_n\}\cup\{w_n\}$. Construct canonical products $P_1(z)$ and $P_2(z)$ with zero 
sequences $\{z_n\}$ and $\{w_n\}$, respectively, and define $P(z)=P_1(z)P_2(z)$. Then
a calculation in the spirit of Example~\ref{q-example}, with $P$ in place of $f$,
shows that $\{\zeta_n\}$ is uniformly $q$-separated. The details are omitted. 
\end{example}


\section{Preparations for the proof of Theorem~\ref{main1}}\label{q-sepa-sec}


The following auxiliary result is a modification of \cite[Corollary 3.3]{HL}
that is needed to find an entire $A(z)$ satisfying the interpolation problem \eqref{interpolation} 
in our proof of Theorem~\ref{main1}. As we see, the growth of such an $A(z)$ depends heavily on 
the uniform $q$-separation of the zeros of $f$. This needs to be taken into account
when proving the inequality $\rho(f)\geq\rho (A)$.

\begin{lemma}\label{interpolation-lemma}
Suppose that $\{z_n\}$ is an infinite sequence of nonzero points in $\C$
with finite exponent of convergence $\lambda$, and that $\{z_n\}$ is uniformly $q$-separated
for some $q\geq 0$ (and $C>0$). Let $\{\sigma_n\}$ be an infinite sequence of points in $\C$,
not necessarily distinct, and let $h:[0,\infty)\to[1,\infty)$ be a continuous and nondecreasing
function such that $|\sigma_n|\leq h(|z_n|)$ for $n\in\N$.

Then there exists an entire function $A(z)\not\equiv 0$ such that
    \begin{equation}\label{int-usualorder}
    A(z_n)=\sigma_n,\quad n\in\N,
    \end{equation}
and, for any given $\alpha>1$,
    \begin{equation}\label{gwth-usualorder}
    \rho(A)\leq\max\left\{\lambda,\limsup_{r\to\infty} \frac{\log I(\alpha r)}{\log r}\right\},
    \end{equation}
where $\displaystyle I(x)=\max_{e\leq t\leq x}\frac{\log\left(h(t)e^{Ct^q}\right)}{\log t}$
is a nondecreasing function for~\mbox{$x\geq e$.}
\end{lemma}

\begin{proof}
Without loss of generality, we may suppose that $\{z_n\}$ is ordered according to increasing moduli.
Moreover, $I(x)$ is well defined and nondecreasing by continuity.
Let $P(z)$ be the canonical product of genus $p\leq \lambda$ having $\{z_n\}$ as its zero sequence.
By uniform $q$-separation, the points $z_n$ are simple, and
hence $\frac{1}{P'(z_n)}\in\C$ for any $n$.

(a) Suppose that $\sigma_n\neq 0$ for all $n$. Let $H(z)$ denote the Mittag-Leffler series in \cite[Eq.~(3.1)]{HL},
i.e.,
    \begin{equation}\label{H0}
 	H(z)=\sum_{n=1}^\infty\frac{c_n}{z-z_n}\left(\frac{z}{z_n}\right)^{q_n},
	\end{equation}
where $c_n=\frac{\sigma_n}{P'(z_n)}$ and $\{q_n\}$ is a sequence such that each $q_n\geq 0$
is the smallest integer satisfying
    \begin{equation*}
    q_n\geq \max\left\{ \alpha\left(\frac{\log |c_n|}{\log |z_n|}+p\right),
    \frac{\log\frac{|c_n|}{n}}{\log|z_n|}+p+1\right\}
	\end{equation*}
for $|z_n|>e$, and set $q_n=0$ otherwise.
Then $H(z)$ is meromorphic in $\C$ and $A=PH$ is an entire function
that satisfies \eqref{int-usualorder}. Hence
it suffices to prove \eqref{gwth-usualorder}. 

By the assumptions, there exists a constant $C_1>0$
such that
	$$
	|c_n|\leq C_1|z_n|h(|z_n|)e^{C|z_n|^q},\quad n\in\N.
	$$
Thus
	\begin{equation}\label{up-usualorder}
	\frac{\log |c_n|}{\log |z_n|}\leq
    \frac{\log\left(h(|z_n|)e^{C|z_n|^q}\right)}{\log |z_n|}+C_2, \quad |z_n|\geq e,
	\end{equation}
where $C_2>0$ is a constant. Set
    $$
    g(t)=\frac{\log\left(h(t)e^{Ct^q}\right)}{\log t}+C_2,\quad t\geq e,
    $$
so that the inequality
    $$
    \frac{\log |c_n|}{\log |z_n|}\leq g(|z_n|)
    $$
holds for every $n$ such that $|z_n|\geq e$. As in the proof of \cite[Corollary 3.3]{HL},
we would like to apply \cite[Theorem~3.1]{HL} next, but the monotonicity of $g(t)$ is not known,
in particular when $0\leq q<1$. Thus we replace $g(t)$ with
    $$
    G(t)=\max_{e\leq x\leq t}g(x),
    $$
which is a non-decreasing function by continuity. Now \cite[Theorem~3.1]{HL} gives us
	$$
	\rho(H)\leq \max\left\{\lambda,\limsup_{r\to\infty} \frac{\log
    G(\alpha r)}{\log r}\right\}.
    $$
Indeed, it is apparent from the proof of \cite[Theorem~3.1]{HL} that the finitely many
indices $n$ for which $|z_n|<e$ have no affect on this conclusion.
The assertion \eqref{gwth-usualorder}
then follows from $\rho(A)\leq \max\{\rho(P),\rho(H)\}$.

(b) Suppose that $\sigma_n=0$ for at least one $n$. If $\sigma_n=0$ for all $n$, we may choose $A(z)=P(z)$,
in which case \eqref{int-usualorder} and \eqref{gwth-usualorder} clearly hold. Hence we may suppose that $\sigma_n=0$ for
at least one index $n$ but not for all $n$. Let $\{s_n\}$ denote the subsequence of $\{\sigma_n\}$ consisting
of all the nonzero points. Let $\{z_n\}=\{\zeta_n\}\cup\{\xi_n\}$ be a partition of the sequence $\{z_n\}$ such
that each $\zeta_n$ corresponds to $s_n$. In other words, the interpolation problem \eqref{int-usualorder} is
transformed into finding an entire function $A(z)$ such that
	\begin{equation}\label{int2-usualorder}
	A(\zeta_n)=s_n\quad\text{and}\quad A(\xi_n)=0.
	\end{equation}
We factorize the canonical product $P(z)$ as $P(z)=R(z)S(z)$, where $R(\zeta_n)=0$ and $S(\xi_n)=0$.
Supposing that $\{s_n\}$ is an infinite sequence, let $H(z)$ denote the Mittag-Leffler series \eqref{H0},
where $c_n=\frac{s_n}{P'(\zeta_n)}\neq 0$. Then $A=PH$ is entire and satisfies \eqref{int2-usualorder}.
The growth condition \eqref{gwth-usualorder} is then proved just as in Part (a).
If $\{s_n\}$ is a finite sequence, then for $H(z)$ we choose a finite Mittag-Leffler series. Once again $A(z)=P(z)H(z)$
is entire and satisfies \eqref{int2-usualorder} along with $\rho(A)=\rho(P)=\lambda$.
\end{proof}

\begin{remark}
(a) The proof of Lemma~\ref{interpolation-lemma} differs from that of \cite[Corollary 3.3]{HL}
in two ways: The possibility that $\sigma_n=0$ for some $n$ is now included, and the growth of 
the majorant function $g(t)$ has been considered in more detail.

(b) Lemma~\ref{interpolation-lemma} will be applied to the target sequence
$\sigma_k=-f''(z_k)/f'(z_k)$ in \eqref{interpolation}, where $f$ is a canonical
product with genus $p\geq 0$. A very similar target sequence
appears in the proof of \cite[Theorem~1]{Bank}. We have the representation
    \begin{equation}\label{rep}
    \sigma_k=-\frac{2p}{z_k}+2\sum_{n\neq k}\left(\frac{z_k}{z_n}\right)^p\frac{1}{z_n-z_k},
    \end{equation}
which was proved for $p=0$ in Example~\ref{Bank-ex}. For $p\geq 1$, we have the first-order
derivatives \eqref{first-derivative2}. A second differentiation yields
    \begin{eqnarray*}
    f''(z_k)=&-&\frac{2p}{z_k^2}e_p(1)\prod_{j\neq k}
    \left(1-\frac{z_k}{z_j}\right)e_p\left(\frac{z_k}{z_j}\right)\\
    &+&\frac{2}{z_k}e_p(1)\sum_{n\neq k}\frac{z_k^p}{z_n^{p+1}}e_p\left(\frac{z_k}{z_n}\right)
    \prod_{j\neq n,k}\left(1-\frac{z_k}{z_j}\right)e_p\left(\frac{z_k}{z_j}\right),
    \end{eqnarray*}
which together with \eqref{first-derivative2} implies \eqref{rep}.

(c) Let $f$ be an entire function having simple zeros at the points $a_n$, and let $\{b_n\}$
be a target sequence. Then the Lagrange interpolation series
    $$
    L(z)=\sum_{n=1}^\infty\frac{b_nf(z)}{f'(a_n)(z-a_n)}
    $$
is an entire solution to the interpolation problem $L(a_n)=b_n$, provided that $L(z)$
converges uniformly in compact subsets of $\C$, see \cite[p.~195]{Levin}. In our case, however,
the target sequence is unbounded, and the zero-sequence $\{a_n\}$ is uniformly $q$-separated.
The latter means in general that $|f'(a_n)|$ can tend to zero exponentially as $n\to\infty$.
Thus it seems unlikely that the
Lagrange interpolation series could be used in proving Theorem~\ref{main1}, and the
use of Lemma~\ref{interpolation-lemma} instead seems to be justified.
\end{remark}


\section{Proof of Theorem~\ref{main1}}\label{proof-main1-sec}


The solution $f$ will be the Lindel\"of function $L_\rho$ of order $\rho>1/2$, which is a canonical product with simple zeros precisely at the points $z_n=-n^{1/\rho}$ ($n\geq 1$) on the negative real axis. We may write
\begin{equation}\label{L}
L_\rho(z)=\prod_{n=1}^\infty\left(1+\frac{z}{n^\alpha}\right)e_p\left(-\frac{z}{n^\alpha}\right),
\end{equation}
where $\alpha=1/\rho$ and $p=\lfloor\rho\rfloor$ ($=$ integer part of $\rho$) is the genus of $L_\rho$.
It is known (see \cite[p.~294]{EF} or \cite[p.~54]{Neva}) that
    $$
    \delta_N(0,L_\rho)=\left\{\begin{array}{rl}
    \displaystyle 1-\frac{|\sin (\pi\rho)|}{q+|\sin(\pi\rho)|},\ & q<\rho\leq q+1/2,\\[12pt]
    \displaystyle 1-\frac{|\sin (\pi\rho)|}{q+1},\ & q+1/2<\rho\leq q+1,
    \end{array}\right.
    $$
where $q\geq 0$ is an integer. It follows that $\delta_N(0,L_\rho)>0$ for every $\rho>1/2$. Since entire functions of order $\leq 1/2$ have no finite Nevanlinna deficient values, the Lindel\"of functions illustrate the sharpness of this inequality.

\begin{lemma}\label{asymptotic-lemma}
For $k,n\in\N$, $k>n$, and $\alpha>0$, we have
    \begin{eqnarray}
	\alpha (k-n)k^{\alpha-1} \leq k^\alpha-n^\alpha\leq\alpha (k-n)n^{\alpha-1},
    \quad &&0<\alpha\leq 1,\label{trivial-estimate}\\[3pt]
    \alpha (k-n)n^{\alpha-1} \leq  k^\alpha-n^\alpha\leq\alpha (k-n)k^{\alpha-1},
    \quad &&1\leq\alpha<\infty. \label{trivial-estimate2}
	\end{eqnarray}
In particular, for every $\alpha>0$, we have $k^\alpha-(k-1)^\alpha\sim \alpha k^{\alpha-1}$ as $k\to\infty$.
\end{lemma}

The crux of the proof of Lemma~\ref{asymptotic-lemma} is a simple identity
    $$
    k^\alpha-n^\alpha=\alpha\int_n^kx^{\alpha-1}\, dx,
    $$
where the integrand $x^{\alpha-1}$ is decreasing for $0<\alpha<1$ and non-decreasing for $\alpha\geq 1$.
We omit the details.

\begin{lemma}\label{target-growth-lem}
Let $L_\rho$ be the Lindel\"of function of order $\rho>1/2$ with zeros $z_k=-k^{1/\rho}$
of genus $p=\lfloor\rho\rfloor$.
Then the (target) sequence $\sigma_k=-L_\rho''(z_k)/L_\rho'(z_k)$
can be written as
    \begin{equation}\label{sigma}
    \sigma_k=\frac{2p}{k^{\alpha}}+2\sum_{n\neq k}\left(\frac{k}{n}\right)^{\alpha p}
    \frac{1}{k^\alpha-n^\alpha},
    \end{equation}
where $\alpha=1/\rho$. Moreover,
    $$
    |\sigma_k|=\left\{\begin{array}{ll}
    O(\log k),\ &p=0,\\[3pt]
    O\left(k^{\alpha(p-1)+1}\log k\right),\ &p\geq 1.
    \end{array}\right.
    $$
\end{lemma}

\begin{proof}
The representation \eqref{sigma} follows immediately from \eqref{rep}. We estimate
the growth of $|\sigma_k|$ in two steps.

(1) Suppose that $p=0$. Then $\rho\in (1/2,1)$, i.e., $\alpha\in (1,2)$. Using \eqref{sigma},
we get
    \begin{eqnarray*}
    |\sigma_k|/2 &\leq & \sum_{n\leq k-1}\frac{1}{k^\alpha-n^\alpha}
    +\sum_{n\geq k+1}\frac{1}{n^\alpha-k^\alpha}=:S_1(k)+S_2(k),
    \end{eqnarray*}
where $S_1(1)=0$.
Suppose that $k\geq 2$. Since $x\mapsto 1/(k^\alpha-x^\alpha)$ is an increasing function
for $x\in [1,k-1]$, the left endpoint rule and \eqref{trivial-estimate2} give
    \begin{eqnarray*}
    S_1(k) &\leq & \int_1^{k-1}\frac{dx}{k^\alpha-x^\alpha}+\frac{1}{k^\alpha-(k-1)^\alpha}\\
    &\leq &
    \int_1^{k-1}\frac{\alpha x^{\alpha-1}dx}{k^\alpha-x^\alpha}+\frac{1}{\alpha(k-1)^{\alpha-1}}
    =O\left(\log k\right).
    \end{eqnarray*}
Analogously, $x\mapsto 1/(x^\alpha-k^\alpha)$ is a decreasing function
for $x\geq k+1$, so the right endpoint rule and \eqref{trivial-estimate2} give
    \begin{eqnarray*}
    S_2(k) &\leq&
    \int_{k+1}^{\infty}\frac{dx}{x^\alpha-k^\alpha}+\frac{1}{(k+1)^\alpha-k^\alpha}\\
    &\leq&\frac{1}{k^{\alpha-1}}\int_{(k+1)/k}^{\infty}\frac{dt}{t^\alpha-1}+\frac{1}{\alpha k^{\alpha-1}},
    \end{eqnarray*}
where we made the change of variable $t=x/k$. It is clear that
    $$
    \int_2^\infty\frac{dt}{t^\alpha-1}\leq\frac{2^\alpha}{2^\alpha-1}\int_2^\infty\frac{dt}{t^\alpha}
    =\frac{2}{(\alpha-1)(2^\alpha-1)}.
    $$
Since $(t^\alpha-1)/(t^2-1)\to \alpha/2$ as $t\to 1^+$, there exists, by continuity, a constant
$\beta\in (0,\alpha/2)$ such that $t^\alpha-1\geq\beta(t^2-1)$
for all $t\in [1,2]$. Therefore
    $$
    \int_{(k+1)/k}^2\frac{dt}{t^\alpha-1}\leq \frac{1}{\beta}\int_{(k+1)/k}^2\frac{dt}{t^2-1}
    =\frac{1}{2\beta}\log\left(\frac{2k+1}{3}\right)\leq \frac{1}{2\beta}\log k.
    $$
This yields $S_2(k)=O\left(k^{1-\alpha}\log k\right)=o(1)$, and, a fortiori, $|\sigma_k|=O(\log k)$.

(2) Suppose that $p\geq 1$. Then $\rho\in [1,\infty)$, i.e., $\alpha\in (0,1]$. Using \eqref{sigma},
we get
    \begin{eqnarray*}
    |\sigma_k|/2 -\frac{p}{k^\alpha}&\leq &
    \sum_{n\leq k-1}\left(\frac{k}{n}\right)^{\alpha p}\frac{1}{k^\alpha-n^\alpha}
    +\sum_{n\geq k+1}\left(\frac{k}{n}\right)^{\alpha p}\frac{1}{n^\alpha-k^\alpha}\\
    &=:& T_1(k)+T_2(k),
    \end{eqnarray*}
where $T_1(1)=0$. Now $p\leq\rho<p+1$, so that
    \begin{equation}\label{ap}
    \alpha p=\frac{p}{\rho}\in\left(\frac{p}{p+1},1\right]
    \quad\textnormal{and}\quad
    \alpha (p+1)=\frac{p+1}{\rho}\in\left(1,\frac{p+1}{p}\right].
    \end{equation}
Suppose that $k\geq 2$. Then \eqref{trivial-estimate} yields
    \begin{eqnarray*}
    T_1(k) &\leq& \sum_{n\leq k-1}\frac{k^{\alpha p}}{k^\alpha-n^\alpha}
    \leq \int_1^{k-1}\frac{k^{\alpha p}}{k^\alpha-x^\alpha}\, dx
    +\frac{k^{\alpha p}}{k^\alpha-(k-1)^\alpha}\\
    &\leq& \frac{k^{\alpha(p-1)+1}}{\alpha}
    \left(\int_1^{k-1}\frac{\alpha x^{\alpha-1}}{k^\alpha-x^\alpha}\, dx+1\right)
    \leq \frac{k^{\alpha(p-1)+1}}{\alpha}
    \left(\log\frac{k}{\alpha}+1\right),
    \end{eqnarray*}
or $T_1(k)=O\left(k^{\alpha(p-1)+1}\log k\right)$.
The function $x\mapsto (x^{\alpha p}(x^\alpha-k^\alpha))^{-1}$ is strictly decreasing
for $x\geq k+1$, so that
    \begin{eqnarray*}
    T_2(k) &\leq& \int_{k+1}^\infty \frac{k^{\alpha p}}{x^{\alpha p}(x^\alpha-k^\alpha)}\, dx
    +\frac{1}{(k+1)^\alpha-k^\alpha}\\
    &\leq& k^{1-\alpha}\int_{(k+1)/k}^\infty
    \frac{dt}{t^{\alpha p}(t^\alpha-1)}+\frac{(k+1)^{1-\alpha}}{\alpha},
    \end{eqnarray*}
where we made the change of variable $t=x/k$.  It is clear that
    $$
    \int_2^\infty\frac{dt}{t^{\alpha p}(t^\alpha-1)}
    \leq\frac{2^\alpha}{2^\alpha-1}\int_2^\infty\frac{dt}{t^{\alpha(p+1)}}
    =C(\alpha,p)<\infty.
    $$
Analogously as in Part (1) of the proof, we obtain
    $$
    \int_{(k+1)/k}^2\frac{dt}{t^{\alpha p}(t^\alpha-1)}
    \leq \int_{(k+1)/k}^2\frac{dt}{t^\alpha-1}\leq\frac{1}{2\beta}\log k
    $$
for some $\beta\in (0,\alpha/2)$. This yields $T_2(k)=O\left(k^{1-\alpha}\log k\right)$.
The desired estimate for $|\sigma_k|$ follows from the estimates for $T_1(k)$ and $T_2(k)$.
\end{proof}

The estimates in the previous proof seem to have some flexibility. Hence the result
is unlikely to be sharp, but it is more than enough for our use nevertheless.

\begin{lemma}\label{ex1a}
Let $L_\rho$ be the Lindel\"of function of order $\rho\in (1/2,\infty)$
with zeros $z_k=-k^{1/\rho}$ of genus $p=\lfloor\rho\rfloor$. Then there exists a
constant $C>0$ such that
    \begin{eqnarray}
    \inf_k\left\{|z_k|e^{C|z_k|^\rho\log |z_k|^\rho}|L_\rho'(z_k)|\right\}&>&0.\label{assertion}
    \end{eqnarray}
In particular, the zero sequence of $L_\rho$
is uniformly $q$-separated for every $q>\rho$.
\end{lemma}

\begin{proof}
Set $\alpha=1/\rho$ for brevity. As the final conclusion is trivial, it suffices
to prove \eqref{assertion}. We do this in two steps.

(1) Suppose that $\rho\in (1/2,1)$, i.e., $\alpha\in (1,2)$.
By appealing to \eqref{first-derivative} and \eqref{L}, we have
	\begin{eqnarray*}
	|z_k||L_\rho'(z_k)|=\prod_{n\leq k-1}\left(\frac{k^\alpha}{n^\alpha}-1\right)
	\prod_{n\geq k+1}\left(1-\frac{k^\alpha}{n^\alpha}\right)=:P_1(k)P_2(k),
	\end{eqnarray*}
where $P_1(1)=1$.
The products $P_1(k)$ and $P_2(k)$ converge for any finite $k$, so it suffices to
estimate them for arbitrary large values of $k$. We will make use
of \eqref{trivial-estimate2} in these estimates.

We estimate $P_1(k)$ downwards by
    \begin{eqnarray*}
    P_1(k)^{-1} &=& \prod_{n=1}^{k-1}\frac{n^\alpha}{k^\alpha-n^\alpha}=\exp\left(
    \sum_{n=1}^{k-1}\log \frac{n^\alpha}{k^\alpha-n^\alpha}\right)
    \leq  \exp\left(\sum_{n=1}^{k-1}\log \frac{n}{\alpha(k-n)}\right)\\
    &=&  \exp\left((k-1)\log\frac{1}{\alpha}+\sum_{n=1}^{k-1}\log \frac{n}{k-n}\right)
    = \exp\left(-(k-1)\log\alpha\right)\leq 1,
    \end{eqnarray*}
where $k\geq 2$.
When estimating $P_2(k)$ downwards, we apply the right endpoint rule to the decreasing function
$x\mapsto \log\left(1+k^\alpha/(x^\alpha-k^\alpha)\right)$ on $[k+1,\infty)$, and make
use of Part (1) in the proof of Lemma~\ref{target-growth-lem}. We conclude that   	
    \begin{eqnarray*}
	P_2(k)^{-1}&=& \prod_{n\geq k+1}\frac{n^\alpha}{n^\alpha-k^\alpha}
	= \exp\left(\sum_{n\geq k+1}\log\left(1+\frac{k^\alpha}{n^\alpha-k^\alpha}\right)\right)\\
    &\leq & \exp\left(\int_{k+1}^\infty \log\left(1+\frac{k^\alpha}{x^\alpha-k^\alpha}\right)\,dx
    +\log \frac{(k+1)^\alpha}{(k+1)^\alpha-k^\alpha}\right)\\
    &\leq & \exp\left(\int_{k+1}^\infty \frac{k^\alpha}{x^\alpha-k^\alpha}\,dx
    +\log \frac{(k+1)^\alpha k}{\alpha k^{\alpha}}\right)\\
    &\leq & \exp\left(k\int_{(k+1)/k}^\infty\frac{dt}{t^\alpha-1}+\log\frac{2^\alpha k}{\alpha}\right)
    \leq \exp\left(C_1k\log k\right)
    \end{eqnarray*}
for some $C_1=C_1(\alpha)>0$ and for all $k\geq 2$.

Finally, we combine the estimates for $P_1(k)$ and $P_2(k)$, and obtain
    \begin{eqnarray*}
    |z_k|e^{C|z_k|^\rho\log |z_k|^\rho}|L_\rho'(z_k)|=e^{Ck\log k}P_1(k)P_2(k)\geq 1,\ && k\geq 2,
    \end{eqnarray*}
where $C$ is any constant satisfying $C\geq C_1$. This completes the proof of \eqref{assertion}
in the case $\rho\in(1/2,1)$.

(2) Suppose that $\rho\in [1,\infty)$, i.e., $\alpha\in (0,1]$.
By appealing to \eqref{first-derivative2} and \eqref{L}, we have
	\begin{eqnarray*}
	\frac{|z_k||L_\rho'(z_k)|}{e_p(1)}
    &=&\prod_{n\leq k-1}\left(\frac{k^\alpha}{n^\alpha}-1\right)
    e_p\left(\frac{k^\alpha}{n^\alpha}\right)
	\prod_{n\geq k+1}\left(1-\frac{k^\alpha}{n^\alpha}\right)
    e_p\left(\frac{k^\alpha}{n^\alpha}\right)\\
    &=:&Q_1(k)Q_2(k),
	\end{eqnarray*}
where $Q_1(1)=1$.
The products $Q_1(k)$ and $Q_2(k)$ converge for any finite $k$, so it suffices to
estimate them for arbitrary large values of $k$. We make use of \eqref{trivial-estimate}
and \eqref{trivial-estimate2} in these estimates.

A trivial elimination of the Weierstrass convergence factor allows us to argue analogously as
in estimating $P_1(k)$ in Part (1) of the proof. In addition, we make use of
the inequality $n^n\leq e^nn!$, which is valid for all $n\in\N$. We have
    \begin{eqnarray*}
    Q_1(k)^{-1} &=& \prod_{n=1}^{k-1}\frac{n^\alpha}{k^\alpha-n^\alpha}
    e_p\left(\frac{k^\alpha}{n^\alpha}\right)^{-1}
    \leq \exp\left(\sum_{n=1}^{k-1}\log\frac{n^\alpha}{k^\alpha-n^\alpha}\right)\\
    &\leq & \exp\left(\sum_{n=1}^{k-1} \log\frac{kn^\alpha}{\alpha(k-n)k^\alpha}\right)\\
    &=&\exp\left((k-1)\log\frac{1}{\alpha}+(1-\alpha)\log\frac{k^{k-1}}{(k-1)!}\right)
    \leq \exp(D_1k),
    \end{eqnarray*}
where $k\geq 2$ and $D_1=\log\frac{1}{\alpha}+1-\alpha\geq 0$ for $0<\alpha\leq 1$.

Before estimating $Q_2(k)$, we make the preliminary manipulations
    \begin{equation}\label{prelim}
    \begin{split}
	Q_2(k)^{-1}&= \prod_{n\geq k+1}\frac{n^\alpha}{n^\alpha-k^\alpha}
    \exp\left(-\sum_{j=1}^p\frac{1}{j}\left(\frac{k^\alpha}{n^\alpha}\right)^j\right)\\
	&= \exp\left(\sum_{n=k+1}^\infty\left(\log\left(\frac{1}{1-k^\alpha/n^\alpha}\right)
    -\sum_{j=1}^p\frac{1}{j}\left(\frac{k^\alpha}{n^\alpha}\right)^j\right)\right)\\
    &= \exp\left(\sum_{n=k+1}^\infty\sum_{j=p+1}^\infty
    \frac{1}{j}\left(\frac{k^\alpha}{n^\alpha}\right)^j\right).
    \end{split}
    \end{equation}
A simple reasoning based on geometric series yields
    \begin{eqnarray*}
	Q_2(k)^{-1}&\leq& \exp\left(\sum_{n=k+1}^\infty \left(\frac{k^\alpha}{n^\alpha}\right)^{p+1}
    \sum_{j=0}^\infty\left(\frac{k^\alpha}{n^\alpha}\right)^j\right)\\
    &=&\exp\left(\sum_{n=k+1}^\infty \left(\frac{k}{n}\right)^{\alpha(p+1)}
    \frac{n^\alpha}{n^\alpha-k^\alpha}\right)\\
    &=&\exp\left(k^\alpha\sum_{n=k+1}^\infty \left(\frac{k}{n}\right)^{\alpha p}
    \frac{1}{n^\alpha-k^\alpha}\right).
    \end{eqnarray*}
The estimate for $T_2(k)$ in the proof of Lemma~\ref{target-growth-lem} then yields
$Q_2(k)^{-1}\leq \exp\left(D_2k\log k\right)$ for some $D_2=D_2(\alpha)>0$ and for all $k\geq 2$.

Finally, we combine the estimates for $Q_1(k)$ and $Q_2(k)$, and obtain
    \begin{equation*}
    |z_k|e^{C|z_k|^\rho\log |z_k|^\rho}|L_\rho'(z_k)|=e^{Ck\log k}Q_1(k)Q_2(k)\geq 1,
    \end{equation*}
where $C$ is any constant satisfying $C>D_1+D_2$ and $k\geq 2$ is large enough.
This completes the proof of \eqref{assertion} in the case $\rho\in [1,\infty)$.
\end{proof}

\medskip
\noindent
\emph{Proof of Theorem~\ref{main1}}.
After these preparations, the actual proof of Theorem~\ref{main1} is now easy. Let
$f=L_\rho$ be the Lindel\"of function of order $\rho\in (1/2,\infty)$, which has simple zeros
and the required growth. In addition, the zeros of $f$ are uniformly $q$-separated
for every $q>\rho$ by Lemma~\ref{ex1a}. For $K\geq 1$, define $h:[0,\infty)\to [1,\infty)$ by
	$$
	h(t)=\left\{\begin{array}{rl}
	K\log (t^\rho+e),\ & p=0,\\
	K\left(t^{\rho(\alpha(p-1)+1)}\log(t^\rho+e)+1\right),\ & p\geq 1.\end{array}\right.
	$$
Then
	$$
	h(|z_k|)=h(k^\alpha)=\left\{\begin{array}{rl}
	K\log (k+e),\ & p=0,\\
	K\left(k^{\alpha(p-1)+1}\log(k+e)+1\right),\ & p\geq 1.\end{array}\right.
	$$
By Lemma~\ref{target-growth-lem}, we may choose $K\geq 1$ so that $|\sigma_k|\leq h(|z_k|)$ for
all $k$. Moreover, for all $p\geq 0$, we have $\log h(t)=O(\log t)$. Thus
	$$
	I(x):=\max_{e\leq t\leq x}\frac{\log\left(h(t)e^{Ct^q}\right)}{\log t}
    \leq \max_{e\leq t\leq x}\log\left(h(t)e^{Ct^q}\right)=Cx^q+O(\log x),
	$$
and we conclude by
Lemma~\ref{interpolation-lemma} that there exists
an entire function $A(z)\not\equiv 0$ satisfying \eqref{interpolation} such that
	\begin{equation}\label{q}	
	\rho(A)\leq\max\{\rho(f),q\}. 
	\end{equation}
The zero sequence $\{z_n\}$ of $f$ that determines $A$ in Lemma~\ref{interpolation-lemma} is fixed, but
it is also uniformly $q$-separated for any $q>\rho(f)$ by Lemma~\ref{ex1a}. Thus the estimate
in \eqref{q} provided by Lemma~\ref{interpolation-lemma} holds for any $q>\rho(f)$, and, 
a fortiori, $\rho(A)\leq \rho(f)$.

Finally, we define the entire coefficient $B(z)$ by \eqref{def-B}, for
which $\rho(B)\leq\rho(A)$ holds by the lemma on the logarithmic derivative. It remains to show
that $B(z)\not\equiv 0$. Suppose on the contrary that $B(z)\equiv 0$, in which case \eqref{ldeAB}
reduces to $f''+A(z)f'=0$. A simple integration shows that
    $$
    f'(z)=\exp\left(\int^z A(\zeta)\, d\zeta\right).
    $$
Since $f$ (and hence $f'$) is of finite order, it follows that $A(z)$ is a polynomial.
This gives us $\rho(f)=\deg(A)+1\geq 1$. Since $\rho(f)\in\N$, it is clear by the definition of
the Lindel\"of function that $p=\rho(f)$. However, according to the representation \eqref{first-derivative2},
the derivative $f'$ should have a $p$-fold zero at the origin, which is a contradiction.
This completes the proof. \hfill$\Box$


\section{Preparations for the proof of Theorem~\ref{main2}}\label{uniform-logarithmic}


The separation of the zeros of the constructed solution $f$ of \eqref{ldeAB}
plays a key role in the proof of Theorem~\ref{main2}.
The zeros of $f$ are uniformly logarithmically $q$-separated (defined below)
for every $q>\rho_{\log}(f)-1$.

For basic properties of
entire (or, more generally, meromorphic) functions of finite logarithmic order, we refer to \cite{BP, Chern}. In
particular,  the logarithmic exponent of convergence of the zeros of an entire $f$ is given by
	$$
	\lambda_{\log}(f)=\limsup_{r\to\infty}\frac{\log n(r)}{\log\log r},
	$$
where $n(r)$ denotes the number of zeros of $f$ in $|z|<r$, counting multiplicities.
Similarly to the usual order, the functions $T(r,f)$ and $\log M(r,f)$
have the same logarithmic order $\rho_{\log}$, and $\rho_{\log}(f')=\rho_{\log}(f)$.
Differing from the usual order, where $\lambda(f)\leq\rho(f)$, we have $\rho_{\log}(f)=\lambda_{\log}(f)+1$.
This reflects the fact that polynomials are of logarithmic order one, and they have
only finitely many zeros.

We require a new concept on point
separation. We say that a sequence $\{z_n\}$
of finite genus $p$ is \emph{uniformly logarithmically $q$-separated}
for $q\geq 0$ provided that there exists a constant $C>0$ such that
	\begin{equation}\label{separation2}
	\inf_{k\in\N}\left\{e^{C(\log(1+|z_k|))^q}
    \prod_{n\neq k}
    \left|1-\frac{z_k}{z_n}\right|\left|e_p\left(\frac{z_k}{z_n}\right)\right|\right\}>0.
	\end{equation}
If $f$ is given by \eqref{product-f}, then using \eqref{first-derivative2},
we may write \eqref{separation2} equivalently as
	$$
	\inf_{k\in\N}\left\{|z_k|e^{C(\log (1+|z_k|))^q}|f'(z_k)|\right\}>0.
	$$
Certainly there exist sequences with zero exponent of convergence which are not uniformly
logarithmically $q$-separated for any $q\geq 0$. Moreover, uniformly logarithmically
$0$-separated sequences are uniformly $0$-separated se\-quences, and vice versa.

We will make use of this new separation concept for sequences of finite logarithmic exponent of convergence,
in which case $p=0$ and the corresponding $f$ reduces to the form \eqref{fun}.
The definition in \eqref{separation2} for general $p$ is given
for possible applications in the future. Moreover, one could replace the polynomial and
logarithmic weights by some monotonic function $\varphi:\R_+\to\R$, and discuss uniform $\varphi$-separation.

In the next example we will discuss the case $\lambda_{\log}=1$ and $q>0$.

\begin{example}\label{exa2}
Let $z_n=2^n$ and $w_n=z_n+\veps_n$, where $q>0$ and
    $$
    \veps_n=\min\left\{1/2, 2^n\exp\left(-(n\log 2)^q\right)\right\},\quad n\geq 1.
    $$
Then $z_n< w_n < z_{n+1}$, $w_{n+1}\geq z_{n+1}\geq \frac{3}{2}w_n$ and $w_n\leq z_n+\frac12\leq\frac43 z_n$. A simple
modification of the reasoning
in Example~\ref{ex1} shows that the combined sequence $\{z_n\}\cup\{w_n\}$ has logarithmic exponent
of convergence equal to one and the sequence is uniformly logarithmically $q$-separated.
\end{example}

To get an analogue of Lemma~\ref{interpolation-lemma}, we first modify
\cite[Theorem~3.1]{HL} to the case of finite logarithmic order. In the following
statement, special attention has been paid to those points $z_n$ that are near the origin,
as well as to minor monotonicity issues in \cite[(3.2)]{HL}.

\begin{lemma}\label{result1}
Suppose the following assumptions hold:
\begin{itemize}
\item[{\rm (a)}] $\{z_n\}$ is a sequence of distinct nonzero points in $\C$ with
            $\lambda_{\log}<\infty$.
\item[{\rm (b)}] $\{c_n\}$ is a sequence of nonzero points in $\C$, not necessarily distinct.
\item[{\rm (c)}] There exists a continuous function $g:[e,\infty)\to [1,\infty)$
			such that $\frac{\log |c_n|}{\log |z_n|}\leq g(|z_n|)$ for all $|z_n|\geq e$.
\item[{\rm (d)}] Given $\alpha >1$, $\{q_n\}$ is a sequence such that each $q_n\geq 0$ is the
           smallest integer satisfying
           \begin{equation}\label{smallest-integer}
           q_n\geq \max\left\{ \alpha\frac{\log |c_n|}{\log |z_n|},
	    \frac{\log\frac{|c_n|}{n}}{\log|z_n|}+1\right\}
		\end{equation}
		for $|z_n|>e$, and set $q_n=0$ otherwise.
\end{itemize}

Then
	\begin{equation}\label{H}
 	H(z)=\sum_{n=1}^\infty\frac{c_n}{z-z_n}\left(\frac{z}{z_n}\right)^{q_n}
	\end{equation}
is meromorphic in $\C$, and has simple poles exactly at the points
$z_{n}$ with residue $c_{n}$. Moreover, we have the growth estimates
	\begin{equation}\label{2inequalities}
	\lambda_{\log}+1\leq \rho_{\log}(H)\leq
	\max\left\{\lambda_{\log}+1, \limsup_{r\to\infty}
	\frac{\log G(\alpha r)}{\log\log r}+1\right\},
	\end{equation}
where $\displaystyle G(x)=\max_{e\leq t\leq x}g(t)$ is a nondecreasing function for $x\geq e$.
\end{lemma}

\begin{proof}
Without loss of generality, we may suppose that $\{z_n\}$ is ordered according to increasing moduli.
Moreover, $G(x)$ is a well defined nondecreasing function by continuity.

Let $\beta=\sqrt[3]{\alpha}$ ($>1$) and $e<R<\infty$. Suppose that $|z|=r\leq
R$, and write
	\begin{equation}\label{H2}
	\begin{split}
	H(z)&=\sum_{|z_n|\leq\beta R}\frac{c_n}{z-z_n}\left(\frac{z}{z_n}\right)^{q_n}
    +\sum_{|z_n|>\beta R}\frac{c_n}{z-z_n}\left(\frac{z}{z_n}\right)^{q_n}\\
	&=:S_1(z)+S_2(z).
	\end{split}
	\end{equation}
The expression $S_1(z)$ in \eqref{H2} is a finite sum, and therefore it
represents a rational meromorphic function in $\C$. Hence, in
order to prove that $H(z)$ is meromorphic in $\C$, it suffices to
show that $S_2(z)$ converges uniformly.
But this can be done analogously as in \cite[p.~293]{HL}
by making use of the fact that the genus of $\{z_n\}$ is $p=0$.

Obviously, all poles $z_n$ of $H(z)$ are simple and have residue $c_n$.

The inequality $\lambda_{\log}+1\leq \rho_{\log}(H)$ being clear
by $N(r,H)\leq T(r,H)$, it remains to
prove the second inequality in \eqref{2inequalities}.
This culminates in estimating $S_1(z)$. As $q_n\geq 0$ is the smallest integer
satisfying \eqref{smallest-integer}, it follows that
    \begin{eqnarray*}
    q_n &\leq& \alpha\left(\frac{\log|c_n|}{\log |z_n|}+1\right)+1\leq
    \alpha\left(g(|z_n|)+1\right)+1\\
    &\leq&  \alpha\left(G(|z_n|)+1\right)+1,\quad |z_n|\geq e.
    \end{eqnarray*}
Proceeding as in \cite[p.~294]{HL} (but replacing $g$ with $G$ to ensure monotonicity),
we can find a constant $C>0$ such that
	$$
	|S_1(z)|\leq Cn\left(\beta^2R\right)(\log R)^\beta R^{\alpha(G(\beta R)+1)},
	$$
provided that $|z|\leq R$ and $|z|\not\in E\cup [0,1]$, where $E\subset[1,\infty)$
has finite logarithmic measure.
Let $P$ be the canonical product associated with the sequence
$\{z_n\}$, and hence $\rho_{\log}(P)=\lambda_{\log}+1$. Then $PH$ is
entire, and
	$$
	|P(z)H(z)|\leq CM(r,P)\left( n\left(\beta^2R\right)(\log R)^\beta
	R^{\alpha(G(\beta R)+1)}+1\right),
	$$
provided that $|z|\leq R$ and $|z|\not\in E\cup [0,1]$. Taking now $R=\beta r$
and applying \cite[Lemma 5]{Gundersen2}, we see that there exists an
$r_0=r_0(\beta)>0$ such that
	$$
	M(r,PH)\leq CM(\beta r,P)\left(n(\beta^4r)
	\left(\log\left(\beta^2r\right)\right)^\beta(\beta^2r)^{\alpha(G(\alpha r)+1)}+1 \right)
	$$
for all $r\geq r_0$. Since $\log n(r)=O\left(\log\log r\right)$, we deduce that
	$$
	\rho_{\log}(PH)\leq \max\left\{\lambda_{\log}+1, \limsup_{r\to\infty}\frac{\log
	G(\alpha r )}{\log\log r}+1\right\}.
	$$
The second inequality in \eqref{2inequalities} follows from
$\rho_{\log}(H)\leq\max\{\rho_{\log}(P),\rho_{\log}(PH)\}$,
which holds by using standard reasoning.
\end{proof}

\begin{lemma}\label{interpolation-lemma2}
Suppose that $\{z_n\}$ is an infinite sequence of nonzero points in $\C$
with $\lambda_{\log}<\infty$, and that $\{z_n\}$ is uniformly logarithmically $q$-separated
for some $q\geq 0$ (and $C>0$). Let $\{\sigma_n\}$ be an infinite sequence of points in $\C$,
not necessarily distinct, and let $h:[0,\infty)\to [1,\infty)$ be a continuous and nondecreasing
function such that $|\sigma_n|\leq h(|z_n|)$ for $n\in\N$.

Then there exists an entire function $A(z)\not\equiv 0$ such that
    \begin{equation}\label{int}
    A(z_n)=\sigma_n,\quad n\in\N,
    \end{equation}
and, for any given $\alpha>1$,
    \begin{equation}\label{gwth}
    \rho_{\log}(A)\leq\max\left\{\lambda_{\log}+1,\limsup_{r\to\infty} \frac{\log
    F(\alpha r)}{\log\log r}+1\right\},
    \end{equation}
where $\displaystyle F(x)=\max_{e\leq t\leq x}\frac{\log\left(h(t)e^{C\left(\log(1+t)\right)^q}\right)}{\log t}$
is a nondecreasing function for~\mbox{$x\geq e$.}
\end{lemma}

\begin{proof}
Without loss of generality, we may suppose that $\{z_n\}$ is ordered according to increasing moduli.
Moreover, $F(x)$ is well defined and nondecreasing by continuity.
Let $P(z)$ be the canonical product having $\{z_n\}$ as its zero sequence.
By uniform logarithmic $q$-separation, the points $z_n$ are simple, and
hence $\frac{1}{P'(z_n)}\in\C$ for any $n$.

(a) Suppose that $\sigma_n\neq 0$ for all $n$. If $H(z)$ denotes the Mittag-Leffler series in \eqref{H},
where $c_n=\frac{\sigma_n}{P'(z_n)}$, then $A=PH$ is entire and satisfies \eqref{int}. Hence
it suffices to prove \eqref{gwth}. By the assumptions, there exists a constant $C_1>0$ such that
	$$
	|c_n|\leq C_1|z_n|h(|z_n|)e^{C\left(\log(1+|z_n|)\right)^q},\quad n\in\N.
	$$
Thus
	\begin{equation}\label{up}
	\frac{\log |c_n|}{\log |z_n|}\leq
    \frac{\log\left(h(|z_n|)e^{C\left(\log(1+|z_n|)\right)^q}\right)}{\log |z_n|}+C_2,\quad |z_n|\geq e,
	\end{equation}
where $C_2>0$ is a constant. Choosing $g(t)=\frac{\log\left(h(t)e^{C\left(\log(1+t)\right)^q}\right)}{\log t}+C_2$
for $t\geq e$, we deduce
by Lemma~\ref{result1} that
	$$
	\rho_{\log}(H)\leq \max\left\{\lambda_{\log}+1,\limsup_{r\to\infty} \frac{\log
    F(\alpha r)}{\log\log r}+1\right\},
    $$
where $\alpha>1$ and $\displaystyle F(x)=
\max_{e\leq t\leq x}\frac{\log\left(h(t)e^{C\left(\log(1+t)\right)^q}\right)}{\log t}$. The assertion \eqref{gwth}
then follows from $\rho_{\log}(A)\leq \max\{\rho_{\log}(P),\rho_{\log}(H)\}$.

(b) Suppose that $\sigma_n=0$ for at least one $n$. If $\sigma_n=0$ for all $n$, we may choose $A(z)=P(z)$,
in which case \eqref{int} and \eqref{gwth} clearly hold. Hence we may suppose that $\sigma_n=0$ for
at least one index $n$ but not for all $n$. But this case is analogous to Part (b) in the proof of
Lemma~\ref{interpolation-lemma}.
\end{proof}


\section{Proof of Theorem~\ref{main2}}\label{proof-main2-sec}


Anderson and Clunie proved the following result \cite[Theorem~2]{AC}: \emph{Given any continuous function $\varphi(r)$
tending monotonically to infinity as $r\to\infty$, no matter how slowly, there exists an
entire function $f$ such that $\log M(r,f)=O\left(\varphi(r)(\log r)^2\right)$ and $\delta_V(0,f)=1$.}

The solution $f$ to the problem above has real zeros with multiplicities tending to infinity. Moreover,
the proof given in \cite{AC} lies upon a technical restriction $\varphi(r)=O(\log r)$, which
yields $\rho_{\log}(f)\in (2,3)$.
Next we present a modification of the Anderson-Clunie reasoning such that
the solution $f$ has only simple zeros, and hence it is suitable for being a solution of \eqref{ldeAB}.
In addition, $f$ has arbitrary pre-given logarithmic order on $(2,\infty)$.

\begin{lemma}\label{AC-example}
For every $\rho\in(2,\infty)$ there exists an entire function $f$ with simple zeros such that
$\rho_{\log}(f)=\rho$ and $\delta_V(0,f)=1$.
\end{lemma}

\begin{proof}
Given $m\in\N$ and $b>0$, define a polynomial
	\begin{equation}\label{polynomial}
	P_{m,b}(z)=\prod_{k=1}^m\left(1-\frac{z}{be^{i\vp_{m,k}}}\right),	
	\end{equation}
where the arguments $\vp_{m,k}$ are chosen such that the zeros $be^{i\vp_{m,k}}$
of $P_{m,b}$ lie symmetrically in $3\pi/4\leq \arg(z)\leq 5\pi/4$
and on the circle $|z|=b$ as follows:
	\begin{eqnarray*}
	m=1\ &:&\ \vp_{1,1}=\pi,\\	
	m=2\ &:&\ \vp_{2,1}=3\pi/4,\ \vp_{2,2}=5\pi/4,\\
	m=3\ &:&\ \vp_{3,1}=3\pi/4,\ \vp_{3,2}=\pi,\ \vp_{3,3}=5\pi/4,\\
	m=4\ &:&\ \vp_{4,1}=3\pi/4,\  \vp_{4,2}=11\pi/12,
	\ \vp_{4,3}=13\pi/12,\ \vp_{4,4}=5\pi/4,
	\end{eqnarray*}
and, in general, for any integer $l\geq 1$,
	\begin{eqnarray*}
	m=2l+1\ &:&\ \vp_{2l+1,i}=\pi-\frac{\pi}{4}\cdot\frac{l-i+1}{l},\ i=1,\ldots,2l+1,\\[3pt]
	m=2l\ &:&\ \vp_{2l,i}=\pi-\frac{\pi}{4}\cdot\frac{2(l-i)+1}{2l-1},\ i=1,\ldots,2l.
	\end{eqnarray*}
If $m$ is even, then all zeros on $|z|=b$ are pairwise complex conjugate numbers. If $m$ is odd, then
precisely one of the zeros on $|z|=b$ lies on the negative real axis, while the rest are pairwise
complex conjugates.
Thus, if $z\in\R$, we have $P_{m,b}(z)\in\R$ due to the elementary identity $\zeta\overline{\zeta}=|\zeta|^2$.
In particular, $P_{m,b}$ is a real polynomial, and $\log P_{m,b}(z)$ is analytic in
$-3\pi/4<\arg(z)<3\pi/4$.

If $P_{m,b}$ has a real zero at $\zeta\in\R_-$, then the modulus
	$$
	\left|1-\frac{z}{\zeta}\right|=\frac{|\zeta-z|}{|\zeta|}
	$$
on any circle $|z|=r$ is maximal when $z=r$. Suppose then that $\zeta,\overline{\zeta}$ is any pair
of conjugate zeros of $P_{m,b}$. Then the modulus
	$$
	\left|\left(1-\frac{z}{\zeta}\right)\left(1-\frac{z}{\overline{\zeta}}\right)\right|
	=\frac{|\zeta-z||\overline{\zeta}-z|}{|\zeta|^2}
	$$
on any circle $|z|=r$ is maximal when $z$ is simultaneously as far as possible
from both points $\zeta,\overline{\zeta}$ that lie in $3\pi/4\leq \arg(w)\leq 5\pi/4$,
that is, when $z=r$. We repeat this reasoning for all zeros of $P_{m,b}$,
and conclude that 	
	\begin{equation}\label{maxmod}
	M(r,P_{m,b})=P_{m,b}(r).
	\end{equation}
Moreover, a simple geometric reasoning yields the growth estimates
    \begin{equation}\label{maxmod2}
    \left(1+\frac{r^2}{b^2}\right)^{m/2}=\left(\frac{\sqrt{b^2+r^2}}{b}\right)^m
    \leq P_{m,b}(r)\leq \left(1+\frac{r}{b}\right)^m.
    \end{equation}

Next we define two nondecreasing sequences $\{b_n\}$ and $\{c_n\}$ of
positive integers by setting $b_1=1=c_1$, and
    $$
    b_n =\exp\left(\left(\sum_{j=1}^{n-1}c_j\right)^{\frac{2}{\rho-2}}\right),\quad	
    c_n = \left\lfloor(\log b_n)^{\rho-1}\right\rfloor
    =\left\lfloor\left(\sum_{j=1}^{n-1}c_j\right)^{2\frac{\rho-1}{\rho-2}}\right\rfloor,
	$$
for $n\geq 2$, where $\lfloor x\rfloor$ denotes the integer part of $x$. This definition corresponds
to the choice of points $b_n,c_n$ in \cite{AC} when $\varphi(r)=(\log r)^{\rho-2}$. We define
a formal canonical product $f$ in terms of the polynomial factors in \eqref{polynomial} as
	\begin{equation}\label{f2}
	f(z)=\prod_{n=1}^\infty P_{c_n,b_n}(z),
	\end{equation}
and prove that this function has the required properties.

For $r>1$ there exists an integer $n$ such that $b_n<r\leq b_{n+1}$. If $n(r)$ denotes the
number of zeros of $f$ in $|z|<r$, counting multiplicities, then
	\begin{equation}\label{counting-function}
	\begin{split}
    n(r) &= n(b_n)=\sum_{j=1}^nc_j=c_n+(\log b_n)^\frac{\rho-2}{2}\\
    &\leq 2(\log b_n)^{\rho-1}\leq 2(\log r)^{\rho-1}.
    \end{split}
	\end{equation}
Thus the zeros of the formal product $f$ have logarithmic exponent of convergence $\leq \rho-1$,
so that $f$ is entire, and, in fact, $\rho_{\log}(f)\leq\rho$, see \cite{Chern}. Let
$\beta=2\frac{\rho-1}{\rho-2}>2$ for short. The auxiliary function
    $$
    g(x)=\frac{(1+x)^\beta}{1+x^\beta},\quad x\geq 0,
    $$
is increasing on $[0,1]$ and decreasing on $[1,\infty)$. Since $g(0)=1$ and since $g(x)\to 1$ as $x\to\infty$, we have $g(x)\geq 1$ for all $x\geq 0$, that is,
    $$
    (1+x)^\beta\geq 1+x^\beta,\quad x\geq 0.
    $$
Therefore, we have $c_1=c_2=1$, $c_3=\lfloor 2^\beta\rfloor\geq \lfloor 2^2\rfloor=4$, and
	\begin{equation}\label{cn-down}
    \begin{split}
	c_n &\geq \left\lfloor(1+c_{n-1})^\beta\right\rfloor\geq \left\lfloor 1+c_{n-1}^\beta\right\rfloor
    \geq c_{n-1}^{\beta}\geq \left\lfloor(1+c_{n-2})^{\beta^2}\right\rfloor\\
	&\geq \left\lfloor 1+c_{n-2}^{\beta^2}\right\rfloor\geq c_{n-2}^{\beta^2}\geq
    \cdots\geq c_3^{\beta^{n-3}}\geq 4^{\beta^{n-3}},\quad n\geq 4.
	\end{split}
    \end{equation}
Let $\mu>\rho-1$, so that $\sigma:=\frac{\mu}{\rho-1}-1>0$. Since
    \begin{eqnarray*}
    &&\sum_{n=4}^\infty \frac{c_n}{(\log b_n)^\mu}
    \leq\sum_{n=4}^\infty \frac{c_n}{\left(\lfloor(\log b_n)^{\rho-1}\rfloor\right)^\frac{\mu}{\rho-1}}
    = \sum_{n=4}^\infty \frac{1}{c_n^\sigma}\leq \sum_{n=4}^\infty 4^{-\sigma\beta^{n-3}}<\infty,\\
    &&\sum_{n=2}^\infty \frac{c_n}{(\log b_n)^{\rho-1}}
    \geq\sum_{n=2}^\infty \frac{c_n}{\lfloor(\log b_n)^{\rho-1}\rfloor+1}
    = \sum_{n=2}^\infty \frac{c_n}{c_n+1}=\infty,
    \end{eqnarray*}
it follows that the logarithmic exponent of convergence of the zero sequence of $f$
is equal to $\rho-1$, and thus $\rho_{\log}(f)=\rho$, see \cite{Chern}.

It remains to prove that $\delta_V(0,f)=1$, which is equivalent to
    $$
    \liminf_{r\to\infty}\frac{N(r,1/f)}{T(r,f)}=0.
    $$
It suffices to show that
    \begin{equation}\label{suffices-deficiency}
    \lim_{n\to\infty}\frac{N(b_n,1/f)}{T(b_n,f)}=0.
    \end{equation}
For $n\geq 2$, we have
    \begin{equation}\label{N-up}
    \begin{split}
    N(b_n,1/f) &= \int_0^{b_n}\frac{n(t)}{t}\, dt\leq n(b_{n-1})\int_1^{b_n}\frac{dt}{t}\\
    &=\left(\sum_{j=1}^{n-1}c_j\right)\log b_n=(\log b_n)^\frac{\rho}{2}.
    \end{split}
    \end{equation}
Estimating the characteristic function downwards requires more work. To begin with,
we observe that the representation \cite[(5.3)]{AC} for $\log M(r,f)$
is not valid in our case. However, we see by means of \eqref{maxmod} and \eqref{maxmod2} that
    $$
    \log M(r,f)=\log f(r)=\sum_{n=1}^\infty \log P_{c_n,b_n}(r)\geq \sum_{n=1}^\infty
    \frac{c_n}{2}\log\left(1+\frac{r^2}{b_n^2}\right).
    $$
For a fixed $r>0$ and all $t>r$, we deduce by \eqref{counting-function} and L'Hospital's rule that
    $$
    0\leq n(t)\log\left(1+\frac{r^2}{t^2}\right)\leq 2\frac{\log \left(1+r^2/t^2\right)}{(\log t)^{-(\rho-1)}}
    \sim \frac{4r^2}{\rho-1}\cdot\frac{(\log t)^\rho}{r^2+t^2}\to 0,
    $$
as $t\to\infty$. Therefore, Riemann-Stieltjes integration and integration by parts give us
	\begin{equation*}
	\log M(r,f)\geq \frac12\int_0^\infty\log\left(1+\frac{r^2}{t^2}\right)\, dn(t)=
    r^2\int_0^\infty \frac{n(t)\, dt}{t(t^2+r^2)}.
	\end{equation*}
Using the well-known inequality \cite[p.~18]{Hayman}
    $$
    \log M(r,f)\leq 3T(2r,f),\quad r>0,
    $$
we then deduce that
    \begin{equation}\label{deduce}
    T(r,f) \geq
    \frac{r^2}{12}\int_0^\infty \frac{n(t)\, dt}{t(t^2+r^2/4)}
    \geq \frac{n(b_n)r^2}{12}\int_{b_n}^{2b_n}\frac{dt}{t(t^2+r^2/4)},
    \end{equation}
where
    $$
    n(b_n)=\sum_{j=1}^nc_j\geq c_n=\left\lfloor(\log b_n)^{\rho-1}\right\rfloor\geq (\log b_n)^{\rho-1}-1.
    $$
Substituting $r=b_n$ and $t=b_nu$ in \eqref{deduce}, it follows that
    $$
    T(b_n,f)\geq \frac{1}{12}\left((\log b_n)^{\rho-1}-1\right)\int_1^2\frac{du}{u(u^2+1/4)}
    \geq \frac{1}{40}\left((\log b_n)^{\rho-1}-1\right).
    $$
Combining this with \eqref{N-up} finally gives \eqref{suffices-deficiency}, because $\rho>2$.
\end{proof}

It is not necessary for $k,n$ in Lemma~\ref{asymptotic-lemma} to be integers.
We will state, without a proof, the following analogue of Lemma~\ref{asymptotic-lemma} needed later on.

\begin{lemma}\label{asymptotic-lemma2}
For $0<A<B$ and $\gamma>0$, we have the inequalities
    \begin{eqnarray}
    \gamma (B-A)B^{\gamma-1}\leq B^\gamma-A^\gamma &\leq& \gamma (B-A)A^{\gamma-1} ,\quad 0<\gamma\leq 1,\label{AB2}\\
    \gamma (B-A)A^{\gamma-1}\leq B^\gamma-A^\gamma &\leq& \gamma (B-A)B^{\gamma-1} ,\quad 1\leq\gamma<\infty.\label{AB}
    \end{eqnarray}
\end{lemma}

It turns out that the zero sequence of the function $f$ in Lemma~\ref{AC-example} is separated
in the following sense.

\begin{lemma}\label{uniformly-zero}
The zero sequence of the function $f$ in \eqref{f2} is uniformly
logarithmically $q$-separated for every $q>\rho-1$.
\end{lemma}

\begin{proof}
Let $\{z_n\}$ denote the zeros of $f$ organized first according to the increasing modulus and
then, on each circle $|z|=b_j$, according to the increasing argument on $[3\pi/4,5\pi/4]$. This
fixes $\{z_n\}$ uniquely. Since $\sum_n |z_n|^{-1}<\infty$ and since $|z_n|\geq e$ for $n\geq 2$, we observe that
    $$
    \prod_{n\geq 2}\left|1-\frac{z_1}{z_n}\right|
    \geq\prod_{n\geq 2}\left(1-\frac{1}{|z_n|}\right)\geq C_0>0.
    $$
If $k\geq 2$, we have $|z_k|=b_i$ for some $2\leq i\leq k$, and
    \begin{equation}\label{P123}
    \begin{split}
    \prod_{n\neq k}\left|1-\frac{z_k}{z_n}\right|
    &=\prod_{|z_n|<b_i}\left|1-\frac{z_k}{z_n}\right|
    \prod_{|z_n|>b_i}\left|1-\frac{z_k}{z_n}\right|
    \prod_{|z_n|=b_i,\, z_n\neq z_k}\left|1-\frac{z_k}{z_n}\right|\\
    &=:P_1(k)P_2(k)P_3(k),
    \end{split}
    \end{equation}
where we set $P_3(2)=1$ because $z_n$ is on the circle $|z|=b_2$ only when $n=2$.
The products $P_1(k),P_2(k),P_3(k)$
converge for any finite $k$, so it suffices to estimate them for arbitrary
large values of $k$ (or for arbitrarily large values of $i$ due to $|z_k|=b_i$).

\medskip
\noindent
\textbf{Estimate for $P_1(k)$:} We have
    \begin{eqnarray*}
    P_1(k) &\geq& \prod_{|z_n|<b_i}\left(\left|\frac{z_k}{z_n}\right|-1\right)
    =\prod_{m=1}^{i-1}\left(\frac{b_i}{b_m}-1\right)^{c_m}
    \geq \prod_{m=1}^{i-1}\left(\frac{b_i}{b_{i-1}}-1\right)^{c_m}.
    \end{eqnarray*}
If $i=2$, we get $P_1(k)\geq e-1>1$, while if $i\geq 3$, we use the inequality $e^x-1\geq x$, and obtain
    \begin{eqnarray*}
    P_1(k) &\geq& \prod_{m=1}^{i-1}\left(\left(\sum_{j=1}^{i-1}c_j\right)^{\gamma}
    -\left(\sum_{j=1}^{i-2}c_j\right)^{\gamma}\right)^{c_m},
    \end{eqnarray*}
where $\gamma=\frac{2}{\rho-2}$. If $\gamma=1$, we clearly have $P_1(k)\geq 1$.
If $\gamma>1$, then \eqref{AB} gives
    $$
    P_1(k)\geq \prod_{m=1}^{i-1}\left(\gamma
    \left(\sum_{j=1}^{i-2}c_j\right)^{\gamma-1}c_{i-1}\right)^{c_m}
    \geq \prod_{m=1}^{i-1}\left(\gamma
    c_1^{\gamma}\right)^{c_m}\geq 1.
    $$
If $0<\gamma<1$, then from \eqref{AB2} we have
    \begin{eqnarray*}
    e^{(\log(1+|z_k|))^q}P_1(k)&\geq&
    e^{(\log b_i)^q}\prod_{m=1}^{i-1}\left(\gamma\left(\sum_{j=1}^{i-1}c_j\right)^{\gamma-1}c_{i-1}\right)^{c_m}\\
    &\geq& \exp\left(\left(\sum_{j=1}^{i-1}c_j\right)^{\gamma q}
    -\sum_{m=1}^{i-1}c_m\log\left(\frac{1}{\gamma}\left(\sum_{j=1}^{i-1}c_j\right)^{1-\gamma}\right)\right).
    \end{eqnarray*}
The function $x\mapsto x^{\gamma q}-x\log\left(\frac{x^{1-\gamma}}{\gamma}\right)$ is eventually increasing
and unbounded for every
$q>1/\gamma$. Summarizing, for any $\gamma>0$ there exists a constant $C_1=C_1(\gamma,q)>0$ such that
    \begin{equation}\label{P1}
    e^{(\log(1+|z_k|))^q}P_1(k)\geq C_1,\quad q> \frac{\rho-2}{2},\ k\geq 2.
    \end{equation}

\smallskip
\noindent
\textbf{Estimate for $P_2(k)$:}
Clearly $j\leq c_j\leq c_{j+1}$ for every $j\geq 3$, and so
    \begin{equation}\label{cn-up}
    \begin{split}
    c_n &\leq \left(\sum_{j=1}^{n-1}c_j\right)^{\beta}\leq \big((n-1)c_{n-1}\big)^{\beta}
    \leq c_{n-1}^{2\beta}\\
    &\leq \left(\sum_{j=1}^{n-2}c_j\right)^{2\beta^2}
    \leq \big((n-2)c_{n-2}\big)^{2\beta^2}\leq c_{n-2}^{(2\beta)^2}\\
    &\leq\ldots
    \leq c_{3}^{(2\beta)^{n-3}}\leq \left(2^\beta\right)^{(2\beta)^{n-3}}\leq 2^{(2\beta)^{n-2}},\quad n\geq 5.
    \end{split}
    \end{equation}
Suppose that $\gamma=1$. For $m\geq i+1\geq 5$, the estimate \eqref{cn-down} yields
    \begin{equation}\label{ratio-b}
    \frac{b_i}{b_{m}}=\exp\left(-\sum_{j=i}^{m-1}c_j\right)
    \leq \exp\big(-c_{m-1}\big)\leq \exp\big(-4^{\beta^{m-4}}\big).
    \end{equation}
In particular, $b_i/b_m\leq 1/e<1/2$ for all $m\geq i+1\geq 5$. Thus, using
$\log(1-x)\geq -2x$ for $x\in \left[0,1/2\right]$, we have
    \begin{eqnarray*}
    P_2(k) &\geq& \prod_{|z_n|>b_i}\left(1-\left|\frac{z_k}{z_n}\right|\right)
    =\prod_{m=i+1}^\infty\left(1-\frac{b_i}{b_m}\right)^{c_m}\\
    &=& \exp\left(\sum_{m=i+1}^\infty c_m\log\left(1-\frac{b_i}{b_m}\right)\right)
    \geq\exp\left(-2\sum_{m=i+1}^\infty c_m\cdot\frac{b_i}{b_m}\right).
    \end{eqnarray*}
By combining \eqref{cn-up} and \eqref{ratio-b}, we deduce that
    \begin{eqnarray*}
    P_2(k)&\geq&
    \exp\left(-2\sum_{m=5}^\infty \exp\left((2\beta)^{m-2}\log 2-4^{\beta^{m-4}}\right)\right).
    \end{eqnarray*}
Since $(2\beta)^{m-2}\log 2-4^{\beta^{m-4}}\leq -4^{\beta^{m-5}}$ for all $m$ large
enough, the series above converges.
Suppose then that $\gamma>1$. We multiply both sides of the first inequality in \eqref{AB} by the
constant $-1$, and use the resulting inequality together with $m\geq i+1$ to deduce that
    \begin{eqnarray*}
    \frac{b_i}{b_{m}} &= &
    \exp\left(\left(\sum_{j=1}^{i-1}c_j\right)^{\gamma}-\left(\sum_{j=1}^{m-1}c_j\right)^{\gamma}\right)\\
    &\leq&
    \exp\left(-\gamma\left(\sum_{j=1}^{i-1}c_j\right)^{\gamma-1}\cdot\sum_{j=i}^{m-1}c_j\right)\\
    &\leq&\exp\left(-\gamma c_1^{\gamma-1}c_{m-1}\right)\leq \exp\left(-c_{m-1}\right).
    \end{eqnarray*}
Next we use \eqref{cn-down} to obtain
	$$
	\frac{b_i}{b_{m}} \leq \exp\big(-4^{\beta^{m-4}}\big),
	$$
and then proceed the same way as in the case $\gamma=1$.
Finally, suppose that $0<\gamma<1$. We multiply both sides of the first inequality in \eqref{AB2}
by the constant $-1$, and use the resulting inequality together with $m\geq i+1$ to deduce that
    \begin{eqnarray*}
    \frac{b_i}{b_{m}} &\leq &
    \exp\left(-\gamma\left(\sum_{j=1}^{m-1}c_j\right)^{\gamma-1}\cdot\sum_{j=i}^{m-1}c_j\right)
    \leq \exp\left(-\gamma c_{m-1}\left(\sum_{j=1}^{m-1}c_j\right)^{\gamma-1}\right)\\
    &\leq& \exp\left(-\gamma c_{m-1}\big((m-1)c_{m-1}\big)^{\gamma-1}\right)
    = \exp\left(-\gamma(m-1)^{\gamma-1}c_{m-1}^{\gamma}\right),
    \end{eqnarray*}
where $m\geq i+1\geq 5$. For $i$ large enough, $b_i/b_m\leq 1/e<1/2$ clearly holds. Similarly as above,
    \begin{eqnarray*}
    P_2(k)
    &\geq& \exp\left(-2\sum_{m=i+1}^\infty c_m\cdot\frac{b_i}{b_m}\right)\\
    &\geq&\exp\left(-2\sum_{m=5}^\infty
    \exp\left((2\beta)^{m-2}\log 2-\gamma (m-1)^{\gamma-1}4^{\gamma\beta^{m-4}}\right)\right).
    \end{eqnarray*}
Since $(2\beta)^{m-2}\log 2-\gamma (m-1)^{\gamma-1}4^{\gamma\beta^{m-4}}\leq -4^{\gamma\beta^{m-5}}$
for all $m$ large enough, the series above converges. Summarizing, for any $\gamma>0$ there
exists a constant $C_2>0$ such that
    \begin{equation}\label{P2}
    P_2(k)\geq C_2,\quad k\geq 2.
    \end{equation}

\smallskip
\noindent
\textbf{Estimate for $P_3(k)$:} Recall that $P_3(k)$ is undefined for $i=1$, and that we have set $P_3(k)=1$
for $i=2$, so suppose that $i\geq 3$. Now $c_i\geq c_3\geq 4$. There are $c_i-1$ factors in $P_3(k)$.
The distance between two consecutive zeros $z_n$ on the circle $|z|=b_i$ is
$2b_i\sin \frac{\pi}{4(c_i-1)}$. Using $\sin x\geq \frac{x}{\sqrt{2}}$ for $0<x<\frac{\pi}{4}$, we get
    \begin{eqnarray*}
    P_3(k)&=&\prod_{|z_n|=b_i,\, z_n\neq z_k}\frac{|z_n-z_k|}{b_i}
    \geq \left(2\sin \frac{\pi}{4(c_i-1)}\right)^{c_i-1}\\
    &\geq& \left(\frac{\sqrt{2}\pi}{4(c_i-1)}\right)^{c_i-1}\geq \left(\frac{1}{c_i-1}\right)^{c_i-1}
    \geq \left(\frac{1}{c_i}\right)^{c_i},
    \end{eqnarray*}
and so
	\begin{eqnarray*}
	e^{(\log(1+|z_k|))^q}P_3(k) &\geq& \exp\left((\log b_i)^q-c_i\log c_i\right)\\
	&\geq& \exp\left((\log b_i)^q-(\rho-1)(\log b_i)^{\rho-1}\log\log b_i\right).
	\end{eqnarray*}
The function $x\mapsto (\log x)^q-(\rho-1)(\log x)^{\rho-1}\log\log x$ is eventually positive for every
$q>\rho-1$. Thus there exists a constant $C_3=C_3(q)>0$ such that
	\begin{equation}\label{P3}
	e^{(\log(1+|z_k|))^q}P_3(k)\geq C_3,\quad q>\rho-1,\ k\geq 2.
	\end{equation}

\smallskip
\noindent
\textbf{Final conclusion:}
We complete the proof of Lemma~\ref{uniformly-zero} by combining \eqref{P123}
with \eqref{P1}, \eqref{P2} and \eqref{P3}.
\end{proof}

\medskip
\noindent
\emph{Proof of Theorem~\ref{main2}}.
After these preparations, the actual proof of Theorem~\ref{main2} is now easy.
Let $f$ be the function in \eqref{f2}. By Lemma~\ref{AC-example}, $f$ has the required properties
for the solution of \eqref{ldeAB}. Further, by Lemma~\ref{uniformly-zero}, the zero sequence
of $f$ is uniformly logarithmically $q$-separated for any $q>\rho-1$.
Since $\rho=\rho_{\log}(f)=\rho_{\log}(f'')$, the target sequence $\{\sigma_n\}$ in \eqref{interpolation}
can be estimated as follows: For every $\sigma>\rho$ there exist constants $C>0$ and $C_1>0$ such that
    $$
    |\sigma_n|\leq C_1|z_n|\exp\big((\log(1+|z_n|))^\sigma+C(\log(1+|z_n|))^q\big),\quad n\in\N.
    $$
Using Lemma~\ref{interpolation-lemma2}, we conclude that there exists
an entire function $A(z)$ satisfying \eqref{interpolation} such that
$\rho_{\log}(A)\leq\max\{\rho_{\log}(f),q\}$, where we may suppose that $q\leq \rho_{\log}(f)$.

Finally, we define the entire coefficient $B(z)$ by \eqref{def-B}.
It is easy to see that $A(z)$ must be transcendental. For if
$A(z)$ is a polynomial, then
    $$
    T(r,B)=m(r,B)=m\left(r,\frac{f''}{f}+A\frac{f'}{f}\right)=O(\log r),
    $$
so that $B(z)$ is also a polynomial.
But this leads to a contradiction as any transcendental solution must be of positive usual order
in the case of polynomial coefficients \cite{G-S-W}. Thus
    $$
    \liminf_{r\to\infty}\frac{T(r,A)}{\log r}=\infty.
    $$
Using the lemma on the logarithmic derivative to $$|B(z)|\leq |f''(z)/f(z)|+|A(z)||f'(z)/f(z)|,$$
we then obtain $\rho_{\log}(B)\leq\rho_{\log}(A)$.
Thus it remains to show that $B(z)\not\equiv 0$. Suppose on the contrary that $B(z)\equiv 0$,
in which case \eqref{ldeAB} reduces to $f''+A(z)f'=0$. Then either $f$ is a polynomial or $\rho(f)\geq 1$,
which are both impossible. This completes the proof. \hfill$\Box$

\bigskip
\noindent
\textbf{Acknowledgement:} All authors were partially supported by the Academy of Finland Project \#268009.
The second author was partially supported by the Magnus Ehrnrooth Foundation
and the Vilho, Yrj\"o and Kalle V\"ais\"al\"a Foundation of the Finnish Academy of Science and Letters.
The third author was partially supported by the National
Natural Science Foundation for the Youth of China (No.~11501402) and the Shanxi Scholarship
Council of China (No.~2015-043).

\end{document}